\documentclass[12pt, leqno]{amsart}

\usepackage{graphicx}
\usepackage{amssymb,amsmath,amsthm,amscd}
\usepackage{mathrsfs}
\usepackage{enumerate}
\usepackage{enumitem}
\usepackage{verbatim}
\usepackage[usenames,dvipsnames]{color}
\usepackage[colorlinks=true, pdfstartview=FitV,
linkcolor=blue,citecolor=blue,urlcolor=blue]{hyperref}
\usepackage[all]{xy}
\usepackage{tikz}
\usepackage{tikz-cd}
\usetikzlibrary{matrix, arrows}
\usepackage{stmaryrd} 
\usepackage{dsfont}
\usepackage{bm} 
\usepackage{scalerel}[2016/12/29] 
\usepackage{amsthm}

\allowdisplaybreaks[4]

\newcommand{\nc}{\newcommand}

\numberwithin{equation}{section}

\usepackage[colorinlistoftodos]{todonotes}
\usepackage{mathtools}
\mathtoolsset{showonlyrefs,showmanualtags}
\usepackage{empheq}

\theoremstyle{plain}
\newtheorem{lemma}{Lemma}[section]
\newtheorem{prop}[lemma]{Proposition}
\newtheorem{theorem}[lemma]{Theorem}
\newcommand{\Prop}{\begin{prop}}
	\newcommand{\enprop}{\end{prop}}
\newcommand{\Lemma}{\begin{lemma}}
	\newcommand{\enlemma}{\end{lemma}}
\newcommand{\Th}{\begin{theorem}}
	\newcommand{\enth}{\end{theorem}}
\newtheorem{corollary}[lemma]{Corollary}
\newcommand{\Cor}{\begin{corollary}}
	\newcommand{\encor}{\end{corollary}}
\newtheorem{definition}[lemma]{Definition}

\newcommand{\Def}{\begin{definition}}
	\newcommand{\edf}{\end{definition}}
\newtheorem{sublemma}[lemma]{Sublemma}
\newcommand{\Sublemma}{\begin{sublemma}}
	\newcommand{\ensub}{\end{sublemma}}
\newtheorem{assum}{Assumption}

\theoremstyle{definition}
\newtheorem{remark}[lemma]{Remark}

\newtheorem{Convention}[lemma]{Convention}
\newcommand{\Conv}{\begin{Convention}}
	\newcommand{\enconv}{\end{Convention}}
\nc{\Rem}{\begin{remark}}
	\nc{\enrem}{\end{remark}}

\newcommand{\arxiv}[1]{\href{http://arxiv.org/abs/#1}{\tt arXiv:\nolinkurl{#1}}}

\nc{\rmkend}{\hfill$\triangledown$}
\nc{\defend}{\hfill$\triangle$}

\newcommand{\on}{\operatorname}
\nc{\be}{\begin{enumerate}}
	\nc{\ee}{\end{enumerate}}
\newcommand{\eq}{\begin{equation}}
	\newcommand{\eneq}{\end{equation}}
\nc{\bc}{\begin{cases}}
	\nc{\ec}{\end{cases}}
\newcommand{\eqn}{\begin{eqnarray*}}
	\newcommand{\eneqn}{\end{eqnarray*}}
\newcommand{\ba}{\begin{array}}
	\newcommand{\ea}{\end{array}}

\nc{\Rp}{\Phi^+}
\nc{\mb}{\underline{m}}
\nc{\acts}{\mathrel{\raisebox{2pt}{$\scaleobj{0.65}{\varolessthan}$}}}
\nc{\lacts}{\mathrel{\hat{\raisebox{2pt}{$\scaleobj{0.65}{\varolessthan}$}}}}
\nc{\lcirc}{\mathrel{\hat{\circ}}}
\nc{\tdelta}{{}^\theta\Delta} 
\nc{\tnabla}{{}^\theta\nabla}
\nc{\tL}{{}^\theta L}
\nc{\wor}{J^\bullet}
\nc{\twor}{{}^\theta J^\bullet}
\nc{\ch}{\operatorname{ch}}
\nc{\gdim}{\operatorname{dim}_q}
\nc{\sh}{\operatorname{sh}}
\nc{\word}{J^\theta}
\nc{\good}{J^\beta_+}
\nc{\tgood}{{}^\theta J^\beta_+}
\nc{\gooda}{J^\bullet_+}
\nc{\tgooda}{{}^\theta J^\bullet_+}
\nc{\tgoodb}{{}^\theta J^\bullet_{+,0}}
\nc{\tgoodc}{{}^\theta J^\bullet_{+,c}}
\nc{\tw}{{}^\theta w}
\nc{\EK}{{}^\theta \mathcal{B}(\mathfrak{g})}
\nc{\Qq}{\mathcal{K}}
\nc{\Ql}{\mathcal{A}}
\nc{\EKm}{{}^\theta V(\lambda)}
\nc{\EKmb}{{}^\theta \mathbf{V}(\lambda)} 
\nc{\EKmz}{{}^\theta V}
\nc{\EKmi}{\EKm_{\Ql}} 
\nc{\EKmiz}{\EKmz_{\Ql}} 
\nc{\tcat}{{}^\theta \gamma}
\nc{\Um}{U_q^-(\mathfrak{g})}
\nc{\EKmiu}{\EKmi^{\mathrm{up}}}
\nc{\EKmil}{\EKmi^{\mathrm{low}}}
\nc{\EKmiuz}{\EKmiz^{\mathrm{up}}}
\nc{\EKmilz}{\EKmiz^{\mathrm{low}}}
\nc{\EKup}{\ttt\mathbf{V}_{\Ql}^{\mathrm{up}}}
\nc{\EKlow}{\ttt\mathbf{V}_{\Ql}^{\mathrm{low}}}
\nc{\oklrn}[1]{{}^\theta \mathcal{R}_{#1}(\lambda)}
\nc{\Fn}{F_i^{(n)}}
\nc{\En}{E_i^{(n)}}
\nc{\Fns}{(F_i^*)^{(n)}}
\nc{\Ens}{(E_i^*)^{(n)}}
\nc{\ff}{\mathbf{f}}
\nc{\ffd}{\mathbf{f}^*}
\nc{\fint}{\mathbf{f}_\Ql}
\nc{\fintd}{\mathbf{f}^*_\Ql}
\nc{\FF}{\mathcal{F}}
\nc{\tFF}{{}^\theta\FF(\lambda)}
\nc{\tFFlow}{{}^\theta\FF(\lambda)_{\Ql}^{\mathrm{low}}}
\nc{\tFFup}{{}^\theta\FF(\lambda)_{\Ql}^{\mathrm{up}}}
\nc{\tFFz}{{}^\theta\FF}
\nc{\FFd}{\mathcal{F}^*}
\nc{\tFFd}{{}^\theta\FFd}
\nc{\Qplus}{\mathsf{Q}_+}
\nc{\Qmin}{\mathsf{Q}_-}
\nc{\tQplus}{\Qplus^\theta} 
\nc{\ttt}{{}^\theta}
\nc{\tchq}{\ttt\operatorname{ch}_q}
\nc{\chq}{\operatorname{ch}_q}
\nc{\ddelta}{\bar{\Delta}}
\nc{\tddelta}{\ttt\ddelta}
\nc{\nnabla}{\bar{\nabla}}
\nc{\tnnabla}{\ttt\nnabla}
\nc{\lyn}{\mathcal{L}_+}
\nc{\tlyn}{\ttt\mathcal{L}_+}
\nc{\front}{\operatorname{front}}
\nc{\back}{\operatorname{back}} 
\nc{\coset}[2]{\mathtt{D}_{#1,#2}}
\nc{\minlet}{\operatorname{minlet}}
\nc{\tcan}[1]{\ttt\mathbf{b}_{#1}}
\nc{\tdcan}[1]{\ttt\mathbf{b}^*_{#1}}
\nc{\tpbw}[1]{\ttt \mathbf{P}_{#1}}
\nc{\tdpbw}[1]{\ttt \mathbf{P}_{#1}^*}
\nc{\proot}{\Phi^+}
\nc{\zodd}{\Z_{\text{odd}}} 
\nc{\nodd}{\N_{\text{odd}}} 
\DeclarePairedDelimiter\norm{\lvert}{\rvert} 
\DeclarePairedDelimiter\tnorm{\ttt\lvert}{\rvert} 
\DeclarePairedDelimiter\Norm{\lVert}{\rVert} 
 
\nc{\zerw}{\ttt\varnothing}
\nc{\Ei}{\mathbf{E}_i}
\nc{\ei}{\mathbf{e}'_i}
\nc{\eis}{\mathbf{e}^*_i} 
\nc{\Lyn}{\mathcal{L}_+}
\nc{\tLyn}{\ttt\mathcal{L}_+}
\nc{\plyn}[1]{\nu^{\langle #1 \rangle}}
\nc{\pplyn}[2]{#1^{\langle #2 \rangle}}
\nc{\tkap}{\ttt\kappa_\nu}
\nc{\EKM}{\ttt\mathbf{V}} 
\nc{\tkpf}{\ttt\operatorname{kpf}} 

\nc{\cor}{\Bbbk}
\nc{\corr}{\mathbf{k}}

\newcommand{\Q}{\mathbb {Q}}
\newcommand{\Z}{{\mathbb Z}}
\newcommand{\N}{{\mathbb N}}

\newcommand{\g}{{\mathfrak{g}}}
\nc{\Ad}{\operatorname{Ad}}

\nc{\sym}{\mathfrak{S}} 
\nc{\weyl}{\mathfrak{W}}
\nc{\tcoset}[2]{\ttt\mathtt{D}_{#1,#2}} 
\nc{\ttau}{{}^\theta\tau}

\newcommand{\Hom}{\operatorname{HOM}}
\newcommand{\hhom}{\operatorname{Hom}}
\newcommand{\End}{\operatorname{End}}
\nc{\Aut}{\operatorname{Aut}}
\newcommand{\Ext}{\operatorname{Ext}}

\nc{\coker}{\operatorname{coker}}

\nc{\Img}{\on{Im}}

\nc{\modv}[1]{{#1}\operatorname{-mod}}
\nc{\gmodv}[1]{{#1}\operatorname{-Mod}}
\nc{\Modv}[1]{{#1}\operatorname{-Mod}}
\nc{\gModv}[1]{{#1}\operatorname{-gMod}}
\nc{\pmodv}[1]{{#1}\operatorname{-pMod}}
\nc{\fmodv}[1]{{#1}\operatorname{-fMod}}

\nc{\Ind}{\on{Ind}}
\nc{\Coind}{\on{Coind}}
\nc{\Res}{\on{Res}}
\nc{\res}{\on{res}}

\nc{\triv}{\mathds{1}}
\nc{\ttriv}{{}^\theta\mathds{1}}
\nc{\tind}[1]{\ttriv \varolessthan {#1}}

\nc{\Ob}{\on{Ob}} 

\nc{\klr}{\mathcal{R}( \beta )}
\nc{\oklr}{{}^\theta \mathcal{R}( \beta; \lambda )}
\nc{\oaklr}{{}^\theta \mathcal{R}( \beta )}
\nc{\klrv}[1]{\mathcal{R}({#1})}
\nc{\klrvv}[2]{\mathcal{R}({#1},{#2})}
\nc{\oklrv}[1]{{}^\theta \mathcal{R}({#1};\lambda)}
\nc{\oklrvv}[2]{{}^\theta \mathcal{R}({#1},{#2};\lambda)} 
\nc{\klrz}[1]{\mathcal{R}_z({#1})}
\nc{\oklrz}[1]{{}^\theta \mathcal{R}_z({#1};\lambda)}

\nc{\comp}{{}^\theta J^\beta}

\nc{\HecB}{\mathcal{H}_{\mathsf{B}_n}(p)}
\nc{\HecBf}{\mathcal{H}^{f}_{\mathsf{B}_n}(p)}

\nc{\bP}{{\mathbb{P}}}
\nc{\bPh}{\widehat{\mathbb{P}}}
\nc{\Oh}{\widehat{\mathcal{O}}} 
\nc{\bK}[1][{n}]{\widehat{\mathbb{K}}_{#1}}

\nc{\tOh}{{}^\theta\widehat{\mathcal{O}}}
\nc{\tOhb}{{}^\theta\widehat{\mathcal{O}}_{\beta}}
\nc{\tbP}{{}^\theta{\mathbb{P}}}
\nc{\tbPh}{{}^\theta\widehat{\mathbb{P}}}
\nc{\tbK}[1][{n}]{{}^\theta\widehat{\mathbb{K}}_{#1}}
\nc{\Kcompl}{{}^\theta\widehat{\mathcal{K}}_\beta} 

\nc{\hV}{\widehat{V}}
\nc{\bV}[1][{n}]{\widehat{V}^{\otimes{#1}}}
\nc{\bVK}[1][{n}]{\widehat{V}^{\otimes{#1}}_{\widehat{\mathbb{K}}}}

\nc{\tbV}[1][{n}]{{}^\theta\widehat{V}^{\otimes{#1}}}
\nc{\tbVK}[1][{n}]{{}^\theta\widehat{V}^{\otimes{#1}}_{\widehat{\mathcal{K}}}}
\nc{\thV}{{}^\theta\widehat{V}}

\nc{\skewr}{{}^\theta\widehat{\mathcal{K}}_\beta \rtimes \cor [ \weyl_n ] }
\nc{\skewrr}{\tbK[\beta] \rtimes \cor [ \weyl_n ] }
\nc{\skewrrr}{\cor [ \weyl_n ] \ltimes \tbK[\beta] }

\nc{\Ulg}{U_q(L\mathfrak{g})}
\nc{\Ugl}{U_q(L\mathfrak{gl}_n)}
\nc{\Usl}{U_q(L\mathfrak{sl}_n)}

\nc{\aff}{{\rm{aff}}}

\nc{\Kmv}[1]{\mathbf{K}_{{#1}}}
\nc{\Rmv}[1]{\mathbf{R}_{{#1}}}
\nc{\Km}{\mathbf{K}}
\nc{\Rm}{\mathbf{R}}

\nc{\tfun}{{}^\theta\mathsf{F}} 
\nc{\fun}{\mathsf{F}} 

\nc{\bl}{\bigl(}
\nc{\br}{\bigr)}
\nc{\lan}{\langle}
\nc{\ran}{\rangle}

\newcommand{\seteq}{\mathbin{:=}}
\nc{\supp}{\on{supp}}
\newcommand{\soplus}{\mathop{\mbox{\normalsize$\bigoplus$}}\limits}
\newcommand{\id}{\on{id}}
\nc{\ord}{\operatorname{ord}}
\nc{\Id}{\operatorname{Id}}
\nc{\gr}{\operatorname{gr}}

\def\vecsign{\mathchar"017E}
\def\dvecsign{\smash{\stackon[-2.335pt]{\vecsign}{\rotatebox{180}{$\vecsign$}}}}
\def\dvec#1{\def\useanchorwidth{T}\stackon[-4.5pt]{#1}{\,\dvecsign}}
\usepackage{stackengine}
\stackMath

\nc{\epito}{\twoheadrightarrow}

\newlength{\mylength}
\DeclareRobustCommand{\SkipTocEntry}[5]{}

\AtBeginDocument{%
   \def\MR#1{}
}

\title[Orientifold KLR algebras and Enomoto--Kashiwara algebra]
{Representations of orientifold Khovanov--Lauda--Rouquier algebras and the Enomoto--Kashiwara algebra}

\author[T. Prze\'{z}dziecki]{Tomasz Prze\'{z}dziecki}
\address{School of Mathematics, University of Edinburgh, Peter Guthrie Tait Rd, Edinburgh, EH9 3FD, United Kingdom}
\email{\href{mailto:tprzezdz@exseed.ed.ac.uk}{tprzezdz@exseed.ed.ac.uk}}



\begin{document}

\begin{abstract}
We consider an ``orientifold" generalization of Khovanov--Lauda--Rou\-quier algebras, depending on a quiver with an involution and a framing. Their representation theory is related, via a Schur--Weyl duality type functor, to Kac--Moody quantum symmetric pairs, and, via a categorification theorem, to highest weight modules over an algebra introduced by Enomoto and Kashiwara. Our first main result is a new shuffle realization of these highest weight modules and a combinatorial construction of their PBW and canonical bases in terms of Lyndon words. Our second main result is a classification of irreducible representations of orientifold KLR algebras and a computation of their global dimension in the case when the framing is trivial. 
\end{abstract}

\maketitle

\setcounter{tocdepth}{1}
\tableofcontents

\section{Introduction}

\addtocontents{toc}{\SkipTocEntry}

Khovanov--Lauda--Rouquier (KLR) algebras were introduced in \cite{Khovanov-Lauda-1, Rouquier-2KM} in the context of  categorification of quantum groups. 
They have since played an increasingly important role in representation theory. 
Broadly speaking, KLR algebras can be regarded, via the Brundan--Kleshchev--Rouquier isomorphism \cite{Brundan-Kleshchev, Rouquier-2KM}, as a generalization of the affine Hecke algebra $\widehat{\mathcal{H}}(\mathtt{A}_m)$ of type $\mathtt{A}$. 
This generalization is twofold. Firstly, KLR algebras naturally possess a non-trivial grading, which is difficult to discern in  the affine Hecke algebra. Secondly, KLR algebras constitute the correct replacement for $\widehat{\mathcal{H}}(\mathtt{A}_m)$  from the point of view of Schur--Weyl duality. Indeed, Kang, Kashiwara and Kim \cite{kang-kashiwara-kim-18} have constructed functors relating categories of modules over KLR algebras and quantum affine algebras of any type, generalizing the relationship between $\widehat{\mathcal{H}}(\mathtt{A}_m)$ and $U_q(\widehat{\mathfrak{sl}}_n)$ established earlier by Chari and Pressley \cite{Chari-Pressley-96}. 

It is natural to ask whether it is possible to construct a KLR-type generalization of affine Hecke algebras of other classical types. A positive answer to this question was given by Vasserot and Varagnolo \cite{VV-HecB} as well as Poulain d'Andecy and Walker \cite{PAW-B}. We will refer to the new graded algebras introduced there as \emph{orientifold KLR algebras} (see Remark \ref{rem: orientifold terminology} for an explanation of the origin of this name). It must be stressed that orientifold KLR algebras are very different from the usual KLR algebras associated to Cartan data of other classical types. From the point of view of categorification, their representation theory is related to an algebra introduced by Enomoto and Kashiwara \cite{Enomoto-Kashiwara-06}, depending on a Dynkin diagram together with an involution. 
More precisely, it was shown in \cite{VV-HecB} that orientifold KLR algebras categorify irreducible highest weight modules $\ttt\mathbf{V}(\lambda)$ over the Enomoto--Kashiwara algebra. In analogy to $U_q(\mathfrak{n}_-)$, these modules also admit a geometric construction in terms of perverse sheaves on the stack of orthogonal representations of a quiver with a contravariant involution \cite{Enomoto}, as well as a Ringel--Hall--type construction \cite{Young-Hall}. 

Our main motivation for studying orientifold KLR algebras is related to Schur--Weyl duality. In \cite{Appel-Przezdziecki-21}, we construct functors between categories of modules over orientifold KLR algebras, and coideal subalgebras $\mathcal{B}_{\mathbf{c},\mathbf{s}}$ of quantum affine algebras $U_q(\widehat{\mathfrak{g}})$ (see \cite{kolb-14}), respectively. The parameter $\lambda$ is related to the parameters ${\mathbf{c},\mathbf{s}}$ via an additional datum in the definition of an orientifold KLR algebra, given by a framing dimension vector. Our intention is to use these functors to develop the graded representation theory of Kac--Moody quantum symmetric pairs. The study of finite-dimensional representations of orientifold KLR algebras is the first step in this programme. 

Let us describe our results in more detail. In \S \ref{sec: oKLR sec2} we introduce a somewhat more general definition of orientifold KLR algebras (Definition \ref{def: oklr}) associated to hermitian matrices with an additional symmetry. We construct a faithful polynomial representation (Proposition \ref{pro: polrep oklr}) and prove a PBW theorem (Proposition \ref{thm: oklr PBW}). \S \ref{sec: EK} is dedicated to the Enomoto--Kashiwara algebra. Inspired by the work of Leclerc \cite{Leclerc-04} and Kleshchev and Ram \cite{Kleshchev-Ram-11}, we construct a shuffle realization of the modules $\ttt\mathbf{V}(\lambda)$ (Definition \ref{def: q shuffle module} and Proposition \ref{pro:shufflemodulereal}). This allows us to apply the combinatorics of Lyndon words to obtain PBW and canonical bases for these modules, in the case $\lambda = 0$ (Theorem \ref{thm: can vs pbw}, Corollary \ref{cor: max of dual can}), somewhat simplifying the original construction of these bases from \cite{Enomoto-Kashiwara-08}. In \S \ref{sec: rep th} we apply these results to the representation theory of orientifold KLR algebras. A key ingredient is Varagnolo and Vasserot's \cite{VV-HecB} categorification theorem, identifying $\ttt\mathbf{V}(\lambda)$ with the Grothendieck group of the category of finite-dimensional representations of orientifold KLR algebras. In our main result (Theorem \ref{thm: main result}), we classify irreducible representations of orientifold KLR algebras in terms of $\theta$-good Lyndon words, and construct them as heads (resp.\ socles) of certain induced (resp.\ coinduced) modules. As an application, we prove that orientifold KLR algebras have finite global dimension when $\lambda = 0$. 

\section*{Future work} 
The present paper lays the foundations for a broader programme connecting the representation theory of quantum symmetric pairs with orientifold KLR algebras via generalized Schur--Weyl duality functors. In \cite{Appel-Przezdziecki-21}, the results of the present paper, together with a number of new techniques, including k-matrices for KLR algebras and localization for module categories, are used to construct Hernandez--Leclerc--type categories (cf.\ \cite{Her-Lec-10, Her-Lec-15}) for coideal subalgebras $\mathcal{B}_{\mathbf{c},\mathbf{s}}$ in affine type $\mathtt{A.III}$ with generic parameters $\mathbf{c},\mathbf{s}$. 

\addtocontents{toc}{\SkipTocEntry}

In future work, we would like to generalize these results to non-generic parameters and coideals of type $\mathtt{D.IV}$. This will, in turn, require the development of the representation theory of orientifold KLR algebras associated to non-trivial framings $\lambda$ and quivers of affine type $\mathtt{D}$. To achieve this, we will combine the combinatorial techniques from the present paper with an in-depth study of the geometry of framed symplectic and orthogonal quiver representations. 

We expect that further study of orientifold KLR algebras with non-trivial framings will also provide new information about the representation theory of (affine) Hecke algebras of types $\mathtt{B}$ and $\mathtt{C}$ with unequal parameters, including the so-called non-asymptotic case, which is still only partially understood.  

In yet another direction, the connection to Hernandez--Leclerc categories suggests that the combinatorics of the dual canonical bases of the modules $\ttt\mathbf{V}(\lambda)$ should have an interesting interpretation in terms of cluster theory. 


\section*{Acknowledgements} 
I would like to thank the University of Edinburgh for excellent working conditions. I am also grateful to A.\ Appel, G.\ Bellamy, P.\ McNamara, V.\ Miemietz, L.\ Poulain d'Andecy, S.\ Rostam, C.\ Stroppel and M.\ Young for discussions and correspondence related to this work. 
Finally, I would like to thank an anonymous referee for a careful reading of the manuscript and many useful suggestions. 
During the writing of this paper I was supported by the starter grant ``Categorified Donaldson-Thomas Theory" No.\ 759967 of the European Research Council. 

\section{Orientifold KLR algebras} \label{sec: oKLR sec2}

\subsection{Some combinatorics} 

Let $\cor$ be a field. 
Let $\sym_n = \langle s_1, \cdots, s_{n-1} \rangle$ denote the symmetric group on $n$ letters, and $\weyl_n = \langle s_0, s_1, \cdots, s_{n-1} \rangle$ the Weyl group of type $\mathtt{B}_{n}$, i.e., $(\Z/2\Z)^n\rtimes\sym_n$. We regard them as Coxeter groups in the usual way. Given $0\leq m \leq n$, let $\coset{m}{n-m}$ (resp.\ $\tcoset{m}{n-m}$) denote the set of shortest left coset representatives with respect to the parabolic subgroup $\sym_m \times \sym_{n-m} \subset \sym_n$ (resp.\ $\weyl_m \times \sym_{n-m} \subset \weyl_n$). 
Let $w_0 \in \sym_n$ (resp.\ $\ttt w_0 \in \weyl_n$) be the longest element, and let ${}^\theta w \in \weyl_{n}$ be the longest element in $\tcoset{0}{n}$, i.e., the signed permutation 
\[
{}^\theta w(l) = -(n-l+1). 
\]

Let $J$ be a set and $\theta \colon J \to J$ an involution. We denote by $J^\theta$ the subset of fixed points of $\theta$. 
Let $\N[J]$ be the commutative semigroup freely generated by $J$. We call elements of $\N[J]$ \emph{dimension vectors}. Given a dimension vector $\beta = \sum_{i \in J} \beta(i) \cdot i$, we set $\Norm{\beta} = \sum_{i \in J} \beta(i)$ and $\supp (\beta) = \{ i \in J \mid \beta(i) \neq 0\}$. 
We call a sequence $\nu = \nu_1 \cdots \nu_n \in J^{n}$ a \emph{composition} of $\beta$ of length $\ell(\nu) = n$ if $\norm{\nu} = \sum_{k=1}^n \nu_k = \beta$. We also set $\Norm{\nu}=n$. Let $J^\beta$ denote the set of all compositions of~$\beta$. There is a left action 
of $\mathfrak{S}_n$ on $J^n$ by permutations
\begin{equation} \label{eq: sym act seq}
	s_k \cdot \nu_1 \cdots \nu_n = \nu_1 \cdots \nu_{k+1} \nu_k \cdots \nu_n \quad (1 \leq k \leq n-1), 
\end{equation} 
whose orbits are the sets $J^\beta$ for all $\beta$ with $\Norm{\beta} = n$. 

Let $\wor = \bigcup_{\beta \in \N[J]} J^\beta$ be the set of compositions of all dimension vectors. We also refer to elements of $\wor$ as \emph{words} in $J$ and denote the empty word by $\varnothing$. 
We consider $\wor$ as a monoid with respect to concatenation: $\nu \mu = \nu_1\cdots \nu_{\ell{\nu}}\mu_1\cdots \mu_{\ell{\mu}}$, with $\varnothing$ as the identity.

The involution $\theta$ induces an involution $\theta \colon \N[J] \to \N[J]$. We call dimension vectors in $\N[J]^\theta$ \emph{self-dual}. 
We will always assume, for any $\beta \in \N[J]^\theta$, that if $i \in J^\theta$ then $\beta(i)$ is even. 
Set $\Norm{\beta}_{\theta} = \Norm{\beta}/2$ and  
\[  {}^\theta(-) \colon \N[J] \to \N[J]^\theta, \quad \beta \mapsto {}^\theta \beta = \beta + \theta(\beta).\] 
We call a sequence $\nu = \nu_1 \cdots \nu_n \in J^n$ a \emph{isotropic composition} of $\beta$ if $\ttt\norm{\nu}=\sum_{k=1}^n {}^\theta\nu_i = \beta$. We abbreviate $\nu_{-k} = \theta(\nu_k)$. Let ${}^\theta J^\beta$ denote the set of all isotropic compositions of $\beta$. 
There is a left action 
of $\weyl_n$ on $J^n$ extending \eqref{eq: sym act seq}, given by 
\[ s_0 \cdot \nu_1 \cdots \nu_n = \theta(\nu_1) \nu_2 \cdots \nu_n,\]
whose orbits are the sets ${}^\theta J^\beta$ for all self-dual $\beta$ with $\Norm{\beta}_\theta = n$. 
Let $\twor = \bigcup_{\beta \in \N[J]^\theta} \comp$ be the set of all isotropic compositions of all self-dual dimension vectors. The identity map defines a bijection $\wor \cong \twor$. 

We will consider algebras depending on matrices and vectors with polynomial entries. Below we introduce some terminology for the latter. 

\begin{definition} 
	We call a matrix $Q = (Q_{ij})_{i,j \in J}$ with entries in $\cor[u,v]$ a \emph{coefficient matrix}. We say that $Q$ is: 
	\begin{enumerate}
		\item[(M1)] \emph{regular} if $Q_{ii}= 0$ for all $i \in J$, 
		\item[(M2)] $\theta$\emph{-symmetric} if $Q_{ij}(u,v) = Q_{\theta(j)\theta(i)}(-v,-u)$ for all $i,j \in J$, 
		\item[(M3)] \emph{non-vanishing} if $Q_{ij} \neq 0$ for all $i \neq j \in J$,
		\item[(M4)] \emph{hermitian} if $Q_{ij}(u,v) = Q_{ji}(v,u)$ for each $i,j \in J$. 
	\end{enumerate} 
	Moreover, we call a vector $Q' = (Q_i)_{i \in J}$ with entries in $\cor[u]$ a \emph{coefficient vector}. We say that $Q'$  is: 
	\begin{enumerate}
		\item[(V1)] \emph{regular} if $Q_i = 0$ for all $i \in J^\theta$, 
		\item[(V2)] \emph{non-vanishing} if $Q_i \neq 0$ for all $i \notin J^\theta$,
		\item[(V3)] \emph{self-conjugate} if $Q_{i}(u) = Q_{\theta(i)}(-u)$. 
	\end{enumerate} 
	If a coefficient matrix satisfies (M1)-(M4), resp.\ if a coefficient vector satisfies (V1)-(V3), we call it \emph{perfect}. 
\end{definition}

\subsection{Reminder on KLR algebras} 

Let $\beta \in \mathbb{N}[J]$ with $\Norm{\beta} = n$, and let $Q$ be a regular coefficient matrix. 

\begin{definition}
	The \emph{KLR algebra} $\mathcal{R}(\beta)$ associated to $(J, Q, \beta)$ is the unital $\cor$-algebra generated by 
	$e(\nu)$ $(\nu \in J^\beta)$, $x_l$ $(1 \leq l \leq n)$ and $\tau_k$ $(1 \leq k \leq n-1)$ subject to the following relations: 
	\begin{itemize}
		\item idempotent relations: 
		\[ 
		e(\nu)e(\nu') = \delta_{\nu,\nu'}e(\nu), \quad 
		x_l e(\nu) = e(\nu) x_l, \quad \tau_k e(\nu) = e(s_k \cdot \nu) \tau_k, 
		\]
		\item polynomial relations: 
		\[
		x_lx_{l'} = x_{l'}x_l, 
		\] 
		\item quadratic relations: 
		\[
		\tau_k^2e(\nu) = Q_{\nu_k,\nu_{k+1}}(x_{k+1},x_k)e(\nu), 
		\]
		\item deformed braid relations: 
		\[
		\tau_k\tau_{k'} = \tau_{k'}\tau_k \ \ \mbox{if} \ \ k \neq k' \pm 1, 
		\]
		\begin{align*}
			(\tau_{k+1}\tau_k\tau_{k+1} - \tau_k\tau_{k+1}\tau_k)e(\nu) =& \ \delta_{\nu_k,\nu_{k+2}} \frac{Q_{\nu_k,\nu_{k+1}}(x_{k+1},x_k) - Q_{\nu_k,\nu_{k+1}}(x_{k+1},x_{k+2})}{x_k - x_{k+2}}e(\nu), 
		\end{align*} 
		\item mixed relations: 
		\begin{align*}
			(\tau_kx_l - x_{s_k(l)}\tau_k)e(\nu) =& \ 
			\left\lbrace
			\begin{array}{ll}
				- e(\nu) & \textrm{ if } \ l=k,\ \nu_k=\nu_{k+1}, \\[0.3em]
				e(\nu) & \textrm{ if } \ l=k+1,\ \nu_k=\nu_{k+1}, \\[0.3em]
				0 & \textrm{ else. } 
			\end{array} 
			\right. 
		\end{align*}
	\end{itemize} 
\end{definition}

Whenever we want to emphasize the dependence of the KLR algebra on the full datum $(J, Q, \beta)$, we will write $\mathcal{R}(J, Q, \beta)$.

\begin{lemma} \label{lem: klr isos for twist}
If the coefficient matrix $Q$ is hermitian, then there is an algebra isomorphism $\mathcal{R}(\beta) \to \mathcal{R}(\beta)$ 
	sending 
	\begin{equation} \label{inv on klr}
		e(\nu) \mapsto e(w_0(\nu)), \quad x_l \mapsto x_{n-l+1}, \quad \tau_k \mapsto -\tau_{n-k}. 
	\end{equation}  
If the coefficient matrix $Q$ is hermitian and $\theta$-symmetric, then there is an algebra isomorphism $\mathcal{R}(\beta) \to \mathcal{R}(\theta(\beta))$ 
	sending 
	\begin{equation} \label{inv on oklr}
		e(\nu) \mapsto e({}^\theta w(\nu)), \quad x_l \mapsto - x_{n-l+1}, \quad \tau_k \mapsto -\tau_{n-k}. 
	\end{equation}  
\end{lemma}

\begin{proof}
The first statement can be found in, e.g., \cite[\S 3.2.1]{Rouquier-2KM}. 
The second statement follows from a direct check of the relations using $\theta$-symmetry. 
\end{proof} 

If $M$ is an $\mathcal{R}(\beta)$-module, we will denote by $M^\dagger$ 
the corresponding $\mathcal{R}(\theta(\beta))$-module 
with the action twisted by the inverse of the isomorphism \eqref{inv on oklr}.

\subsection{Orientifold KLR algebras}

Let $\beta \in \N[J]^\theta$ with $\Norm{\beta}_{\theta} = n$, let $Q$ be a regular $\theta$-symmetric coefficient matrix 
and $Q'$ a regular coefficient vector. 

\begin{definition} \label{def: oklr}
	We define the \emph{orientifold KLR algebra} $\oaklr$ associated to $(J, \theta, Q, Q', \beta)$ to be the unital $\cor$-algebra generated by $e(\nu)$ $(\nu \in \comp)$, $x_l$ $(1 \leq l \leq n)$, $\tau_0$ and $\tau_k$ $(1 \leq k \leq n-1)$ subject to the following relations: 
	\begin{itemize}
		\item idempotent relations: 
		\[ 
		e(\nu)e(\nu') = \delta_{\nu,\nu'}e(\nu), \quad 
		x_l e(\nu) = e(\nu) x_l, \quad \tau_k e(\nu) = e(s_k \cdot \nu) \tau_k, \quad \tau_0 e(\nu) = e(s_0 \cdot \nu) \tau_0, 
		\]
		\item polynomial relations: 
		\[
		x_lx_{l'} = x_{l'}x_l, 
		\] 
		\item quadratic relations: 
		\[
		\tau_k^2e(\nu) = Q_{\nu_k,\nu_{k+1}}(x_{k+1},x_k)e(\nu), \quad \tau_0^2 e(\nu) = Q_{\nu_1}(-x_1) e(\nu), 
		\]
		\item deformed braid relations: 
		\[
		\tau_k\tau_{k'} = \tau_{k'}\tau_k \ \ \mbox{if} \ \ k \neq k' \pm 1, \quad \quad \tau_0\tau_k = \tau_k \tau_0 \ \ \mbox{if} \ \ k \neq 1, 
		\]
		\[
			(\tau_{k+1}\tau_k\tau_{k+1} - \tau_k\tau_{k+1}\tau_k)e(\nu) = \delta_{\nu_k,\nu_{k+2}} \frac{Q_{\nu_k,\nu_{k+1}}(x_{k+1},x_k) - Q_{\nu_k,\nu_{k+1}}(x_{k+1},x_{k+2})}{x_k - x_{k+2}}e(\nu), \]
\begin{align*}
			&((\tau_1\tau_0)^2 - (\tau_0\tau_1)^2)e(\nu) = \\ \\ 
			&=\left\lbrace
			\begin{array}{ll} 
				\frac{Q_{\nu_2}(x_2) - Q_{\nu_1}(x_1)}{x_1+x_2}\tau_1e(\nu) & \textrm{ if } \ \nu_1\neq \nu_2, \ \nu_2 = \theta(\nu_1) \\ \\
				\frac{Q_{\nu_1,\nu_2}(x_2,-x_1) - Q_{\nu_1,\nu_2}(-x_2,-x_1)}{x_2} \tau_0 e(\nu) & \textrm{ if } \ \nu_1\neq\theta(\nu_1), \ \nu_2 = \theta(\nu_2), \\ \\ 
				\frac{Q_{\nu_1,\nu_2}(x_2,-x_1)-Q_{\nu_1,\nu_2}(x_2,x_1)}{x_1x_2} (x_1\tau_0+1)e(\nu) & \textrm{ if } \ \theta(\nu_1) = \nu_1 \neq \nu_2=\theta(\nu_2),  \\ \\
				0 & \textrm{ else,} 
			\end{array}\right. 
		\end{align*}
		\item mixed relations: 
		\begin{align*}
			(\tau_kx_l - x_{s_k(l)}\tau_k)e(\nu) =& \ 
			\left\lbrace
			\begin{array}{ll}
				- e(\nu) & \textrm{ if } \ l=k,\ \nu_k=\nu_{k+1}, \\[0.3em]
				e(\nu) & \textrm{ if } \ l=k+1,\ \nu_k=\nu_{k+1}, \\[0.3em]
				0 & \textrm{ else, } 
			\end{array} 
			\right. \\ \\
			(\tau_0x_1 + x_{1}\tau_0)e(\nu) =& \ 
			\left\lbrace
			\begin{array}{ll}
				0 & \textrm{ if } \ \nu_1 \neq \theta(\nu_1), \\[0.3em]
				-2 e(\nu) & \textrm{ if } \ \nu_1 = \theta(\nu_1).  
			\end{array} 
			\right. \\
			\tau_0x_l =& \ x_l \tau_0 \ \ \mbox{if} \ \ l \neq 1, 
		\end{align*} 
	\end{itemize}
\end{definition} 

By convention, we set ${}^\theta\mathcal{R}(0) = \cor$. Whenever we want to emphasize the dependence of the orientifold KLR algebra on the full datum $(J, \theta, Q, Q', \beta)$, we will write ${}^\theta\mathcal{R}(J, Q, Q', \beta)$. 

\begin{remark} \label{rem: orientifold terminology}
In the case when the matrices $Q,Q'$ arise from a quiver with a contravariant involution and a framing (cf. \S \ref{ss: quiver klr}), under the assumption that the involution has no fixed points, the algebra $\oaklr$ was introduced by Varagnolo and Vasserot \cite{VV-HecB}. The case of an involution with possible fixed points was first considered by Poulain d'Andecy and Walker in \cite{PAW-B}, and later by Poulain d'Andecy and Rostam in \cite{PAR}. The latter paper takes a somewhat similar approach to ours - the definition of the algebra depends on polynomials $Q_{ij}$, but they are less general than ours, and the polynomials $Q_i$ are absent. 

In the literature these algebras are typically referred to as ``generalizations of KLR algebras for types $\mathtt{BCD}$". However, we feel that this name may lead to confusion between, for example, the algebra $\oaklr$ and the KLR algebra $\klr$ associated to a quiver of type $\mathtt{D}$. To avoid this confusion, we propose to introduce the name ``orientifold KLR algebras'' for $\oaklr$. The motivation comes from the connection with orientifold Donaldson-Thomas theory, see \cite{Przez-coha, Young-coha}. 
\end{remark}

\begin{prop} 
	We list several isomorphisms between orientifold KLR algebras. 
	\begin{enumerate}
		\item If $Q$ is hermitian and $Q'$ self-conjugate, then there is an algebra automorphism 
		\eq \label{inv on oklr2} \oaklr \xrightarrow{\sim} \oaklr, \qquad e(\nu) \mapsto e(\ttt w_0(\nu)), \ x_l \mapsto -x_l, \ \tau_k \mapsto -\tau_k \eneq
		with $\nu \in \comp$, $1 \leq l \leq n$ and $0 \leq k \leq n - 1$. 
		\item If $Q$ is hermitian and $Q'$ self-conjugate, then there is an algebra isomorphism
		\begin{equation} \label{anti-invo} \omega \colon \oaklr \xrightarrow{\sim} \oaklr^{op}, \quad e(\nu) \mapsto e(\nu), \ x_le(\nu) \mapsto x_l e(\nu), \ \tau_k e(\nu) \mapsto \tau_k e(s_k\cdot\nu).\end{equation}
		\item Given $\{\zeta_i\}_{i\in J}$ in $\cor$ satisfying $\zeta_i = - \zeta_{\theta(i)}$, as well as $\{\eta_{ij}\}_{i,j\in J}$ and $\{\eta_i\}_{i \in J}$ in $\cor^{\times}$, satisfying: $\eta_{ij} = \eta_{\theta(j)\theta(i)}$ for all $i,j\in J$ and $\eta_i = \eta_{ii}$ for $i \in J^\theta$, 
		let $\hat{Q}_{ij}(u,v) = \eta_{ij}\eta_{ji}(\eta_{jj}u + \zeta_j, \eta_{ii}v + \zeta_i)$ and $\hat{Q}_i(u) = \eta_i \eta_{\theta(i)} Q_i(\eta_{ii}u - \zeta_i)$. Then there is an algebra isomorphism $  {}^\theta\mathcal{R}(J, \hat{Q}, \hat{Q}', \beta)\xrightarrow{\sim}{}^\theta\mathcal{R}(J, Q, Q', \beta)$ given by  
		\[ e(\nu) \mapsto e(\nu), \ x_l e(\nu) \mapsto \eta_{\nu_l,\nu_l}^{-1}(x_l - \zeta_{\nu_l})e(\nu), \ \tau_k e(\nu) \mapsto \eta_{\nu_k,\nu_{k+1}}\tau_k e(\nu), \ \tau_0 e(\nu) \mapsto \eta_{\nu_1} \tau_0 e(\nu). 
		\]
	\end{enumerate}
\end{prop}

\begin{proof}
	The proposition follows by a direct computation from the defining relations. 
\end{proof}

\subsection{Polynomial representation} 
Set 
\begin{alignat*}{6}
	\bP_{ \nu } :=& \ \cor [x_1, \ldots, x_n] e(\nu), &\quad \bPh_{ \nu } :=& \ \cor [[x_1, \ldots, x_n]] e(\nu), &\quad
	\widehat{\mathbb{K}}_\nu :=& \ \cor((x_1, \ldots, x_n)) e(\nu), \\
	\tbP_{ \beta } \ \seteq& \soplus_{\nu \in \comp} \bP_{ \nu }, &\quad
	\tbPh_{ \beta } \ \seteq& \soplus_{\nu \in \comp} \bPh_{ \nu }, &\quad
	{}^\theta\widehat{\mathbb{K}}_{ \beta } \ \seteq& \soplus_{\nu \in \comp} \widehat{\mathbb{K}}_{ \nu }. 
\end{alignat*}
We abbreviate $x_{-l} = - x_l$ for $1 \leq l \leq n$. 
The group $\weyl_n$ acts on $\cor((x_1,\hdots,x_n))$ from the left by $w\cdot x_l = x_{w(l)}$. This induces an action on ${}^\theta\widehat{\mathbb{K}}_{ \beta }$ according to the rule 
\begin{equation} \label{eq: weyl action klr} w \cdot f e(\nu) = w(f) e(w\cdot \nu), \end{equation} 
for $w \in \weyl_n$ and $f \in \cor((x_1,\hdots,x_n))$. 

Let $P = (P_{ij})_{i,j\in J}$ be a coefficient matrix satisfying (M1)-(M3) and $P'=(P_i)_{i \in J}$ a coefficient vector satisfying (V1)-(V2). Set 
\eq \label{eq: P->Q} Q_{ij}(u,v) = P_{ij}(u,v)P_{ji}(v,u), \quad Q_i(u) = P_i(u)P_{\theta(i)}(-u), \qquad (i,j \in J). \eneq 
Then $Q = (Q_{ij})$ is a perfect coefficient matrix and $Q' = (Q_i)$ a perfect coefficient vector. 

\begin{prop} \label{pro: polrep oklr}
	The algebra $\oaklr$ has a faithful polynomial representation on $\tbP_{ \beta }$, given by: 
	\begin{itemize} 
		\item $e(\nu)$ $(\nu \in \comp)$ acting as projection onto $\bP_\nu$, 
		\item $x_1, \hdots, x_n$ acting naturally by multiplication, 
		\item $\tau_1, \hdots, \tau_{n-1}$ acting via  
		\[ 
		\tau_k \cdot f e(\nu) = 
		\left\lbrace
		\begin{array}{ll}
			\displaystyle\frac{s_k(f)-f}{x_{k}-x_{k+1}}e(\nu) & \textrm{ if } \nu_k=\nu_{k+1}, \\ \\ 
			P_{\nu_k,\nu_{k+1}}(x_{k},x_{k+1})s_k( f)e\bigl(s_k\cdot \nu \bigr) & \textrm{ otherwise, } \\ \\
		\end{array} 
		\right. 
		\]
		\item $\tau_0$ acting via 
		\[
		\tau_0 \cdot f e(\nu) = 
		\left\lbrace
		\begin{array}{ll}
			\displaystyle\frac{s_0(f)-f}{x_1}e(\nu) & \textrm{ if } \theta(\nu_1) = \nu_1, \\ \\
			P_{\nu_1}(x_1)s_0(f)e\bigl(s_0\cdot \nu \bigr) & \textrm{ otherwise. } 
		\end{array}
		\right. 
		\]
	\end{itemize} 
\end{prop}

Whenever we want to emphasize the dependence of the above representation on $(P,P')$, we will write $\tbP_{\beta}^{P,P'}$. 

\begin{proof}
	The proof that the operators defined above satisfy all the relations from Definition \ref{def: oklr} not involving $\tau_0$ is the same as in the case of the KLR algebra, and can be found in, e.g., the proof of \cite[Proposition 3.12]{Rouquier-2KM}. The other relations are easy to check, with the exception of the deformed braid relations. We prove these explicitly below. 
	
	To simplify exposition, we omit the idempotents. We also abbreviate $i = \nu_1$, $j = \nu_2$. First consider the case where $i \neq j$ and $j = \theta(i)$.  Then: 
	
	\begin{align*}
		\tau_1\tau_0\tau_1 \tau_0(f) =& \ \tau_1\tau_0\tau_1 P_i(x_1)s_0(f) = \tau_1\tau_0\frac{P_i(x_2)s_1s_0(f) - P_i(x_1)s_0(f)}{x_1-x_2} \\
		=& \ \tau_1 P_j(x_1) \frac{P_i(x_2)s_0s_1s_0(f) - P_i(-x_1)f}{-x_1-x_2} \\
		=& \ P_{ij}(x_1,x_2)P_j(x_2) \frac{P_i(x_1)s_1s_0s_1s_0(f) - P_i(-x_2)s_1(f)}{-x_1-x_2}, 
	\end{align*}
	
	\begin{align*}
		\tau_0\tau_1\tau_0\tau_1(f) =& \ \tau_0\tau_1\tau_0 P_{ij}(x_1,x_2)s_1(f) = \tau_0\tau_1 P_j(x_1)P_{ij}(-x_1,x_2)s_0s_1(f) \\
		=& \ \tau_0 \frac{P_j(x_2)P_{ij}(-x_2,x_1)s_1s_0s_1(f) - P_j(x_1)P_{ij}(-x_1,x_2)s_0s_1(f)}{x_1-x_2} \\
		=& \ P_i(x_1) \frac{P_j(x_2)P_{ij}(-x_2,-x_1)s_0s_1s_0s_1(f) - P_j(-x_1)P_{ij}(x_1,x_2)s_1(f)}{-x_1-x_2}.
	\end{align*}
	Since, by $\theta$-symmetry, we have $P_{ij}(x_1,x_2) = P_{ij}(-x_2,-x_1)$, it follows that 
	\[ \textstyle 
	((\tau_1\tau_0)^2 - (\tau_0\tau_1)^2)(f) = \frac{P_j(x_2)P_i(-x_2) - P_i(x_1)P_j(-x_1)}{x_1+x_2}P_{ij}(x_1,x_2)s_1(f) = \frac{Q_j(x_2)-Q_i(x_1)}{x_1+x_2}\tau_1(f). 
	\] 
	
	Secondly, let $i \neq \theta(i)$ and $j = \theta(j)$. Then: 
	
	\begin{align*}
		\tau_1\tau_0\tau_1 \tau_0(f) =& \ \tau_1\tau_0\tau_1 P_i(x_1)s_0(f) = \tau_1\tau_0 P_{\theta(i),j}(x_1,x_2)P_i(x_2) s_1s_0(f) \\
		=& \ \tau_1 \frac{P_{\theta(i),j}(-x_1,x_2)P_i(x_2) s_0s_1s_0(f) - P_{\theta(i),j}(x_1,x_2)P_i(x_2) s_1s_0(f)}{x_1} \\
		=& \textstyle \ P_{j,\theta(i)}(x_1,x_2) \frac{P_{\theta(i),j}(-x_2,x_1)P_i(x_1) s_1s_0s_1s_0(f) - P_{\theta(i),j}(x_2,x_1)P_i(x_1) s_0(f)}{x_2},
	\end{align*}
	
	\begin{align*}
		\tau_0\tau_1\tau_0\tau_1(f) =& \ \tau_0\tau_1\tau_0 P_{ij}(x_1,x_2)s_1(f) = \tau_0\tau_1 \frac{P_{ij}(-x_1,x_2)s_0s_1(f) - P_{ij}(x_1,x_2)s_1(f)}{x_1} \\
		=& \ \tau_0 P_{ji}(x_1,x_2)\frac{P_{ij}(-x_2,x_1)s_1s_0s_1(f) - P_{ij}(x_2,x_1)f}{x_2} \\
		=& \ P_i(x_1) P_{ji}(-x_1,x_2) \frac{P_{ij}(-x_2,-x_1)s_0s_1s_0s_1(f) - P_{ij}(x_2,-x_1)s_0(f)}{x_2}. 
	\end{align*}
	Again, $\theta$-symmetry implies that 
	\begin{align*}
		((\tau_1\tau_0)^2 - (\tau_0\tau_1)^2)(f) =& \textstyle \ \frac{-P_{j,\theta(i)}(x_1,x_2)P_{\theta(i),j}(x_2,x_1) + P_{ij}(x_2,-x_1)P_{j,i}(-x_1,x_2)}{x_2}P_{i}(x_1)s_0(f) \\
		=& \ \frac{Q_{ij}(x_2,-x_1)-Q_{ij}(-x_2,-x_1)}{x_2}\tau_0(f). 
	\end{align*}
	
	Thirdly, let $\theta(i) = i \neq j = \theta(j)$. Then: 
	\begin{align*}
		\tau_1\tau_0\tau_1 \tau_0(f) =& \ \tau_1\tau_0\tau_1 \frac{s_0(f) - f}{x_1} = \tau_1\tau_0 P_{ij}(x_1,x_2) \frac{s_1s_0(f) - s_1(f)}{x_2} \\
		=& \ \tau_1 \frac{P_{ij}(-x_1,x_2)[s_0s_1s_0(f) - s_0s_1(f)] - P_{ij}(x_1,x_2)[s_1s_0(f) - s_1(f)]}{x_1x_2} \\
		=& \textstyle \ P_{ji}(x_1,x_2) \frac{P_{ij}(-x_2,x_1)[s_1s_0s_1s_0(f) - s_1s_0s_1(f)] - P_{ij}(x_2,x_1)[s_0(f) - (f)]}{x_1x_2}, 
	\end{align*}
	
	\begin{align*}
		\tau_0\tau_1\tau_0\tau_1(f) =& \ \tau_0\tau_1\tau_0 P_{ij}(x_1,x_2)s_1(f) = \tau_0\tau_1 \frac{P_{ij}(-x_1,x_2)s_0s_1(f) - P_{ij}(x_1,x_2)s_1(f)}{x_1} \\
		=& \ \tau_0 P_{ji}(x_1,x_2) \frac{P_{ij}(-x_2,x_1)s_1s_0s_1(f) - P_{ij}(x_2,x_1)f}{x_2} \\ 
		=& \ \frac{P_{ji}(-x_1,x_2)[P_{ij}(-x_2,-x_1)s_0s_1s_0s_1(f)-P_{ij}(x_2,-x_1)s_0(f)]}{x_1x_2} \\
		-& \ \frac{P_{ji}(x_1,x_2)[P_{ij}(-x_2,x_1)s_1s_0s_1(f) - P_{ij}(x_2,x_1)f]}{x_1x_2}. 
	\end{align*}
	By $\theta$-symmetry, we conclude that 
	\begin{align*}
		((\tau_1\tau_0)^2 - (\tau_0\tau_1)^2)(f) =& \ \frac{-P_{ji}(x_1,x_2)P_{ij}(x_2,x_1) + P_{ji}(-x_1,x_2)P_{ij}(x_2,-x_1)}{x_1x_2}s_0(f) \\
		=& \ \frac{Q_{i,j}(x_2,-x_1)-Q_{i,j}(x_2,x_1)}{x_1x_2} (x_1\tau_0+1)f. 
	\end{align*} 
	
	Fourthly, let $i = \theta(i)$ and $j \neq \theta(j)$. One easily checks (using $\theta$-symmetry) that 
	$((\tau_1\tau_0)^2 - (\tau_0\tau_1)^2)(f) = g\cdot s_1s_0s_1\Delta_0(f) - \Delta_0(g\cdot s_1s_0s_1(f))$, where $g$ is an $s_0$-invariant polynomial and $\Delta_0 = x_1^{-1}(s_0 - 1)$. It now follows from the properties of Demazure operators that 
	\[
	((\tau_1\tau_0)^2 - (\tau_0\tau_1)^2)(f) = g\cdot s_1s_0s_1\Delta_0(f) - (\Delta_0(g)\cdot s_1s_0s_1(f) + s_0(g)\Delta_0(s_1s_0s_1(f))) = 0. 
	\]
	
	Fifthly, let $i=j$ and $i \neq \theta(i)$. One checks, as above, that 
	$((\tau_1\tau_0)^2 - (\tau_0\tau_1)^2)(f) = \Delta_1(g\cdot s_0s_1s_0(f)) - g\cdot s_0s_1s_0\Delta_1(f)$, where $g$ is an $s_1$-invariant polynomial and $\Delta_1 = (x_1-x_2)^{-1}(s_1 - 1)$. As above, it follows from the properties of Demazure operators that $((\tau_1\tau_0)^2 - (\tau_0\tau_1)^2)(f)~=~0$. 
	
	Finally, suppose that $i = j = \theta(j)$. Then each $\tau_0$ and $\tau_1$ acts as a Demazure operator, but Demazure operators satisfy the braid relation. This completes the proof that $\tbP_{ \beta }$ is a representation of $\oaklr$. 
	
	The proof of faithfulness is analogous to the case of KLR algebras, see, e.g., \cite[Proposition 3.12]{Rouquier-2KM}. 
\end{proof} 

Next, for each $i,j \in J $,
we choose holomorphic functions $c_{ij}(u,v) \in\cor[[u,v]]$ such that
\begin{equation} \label{eq: c-ij} 
	c_{ij}(u,v)c_{ji}(v,u)=1, \quad c_{ii}(u,v)=1, \quad c_{ij}(u,v) = c_{\theta(j)\theta(i)}(-v,-u).  
\end{equation}
Moreover, for each $i \in J$, we also choose holomorphic functions $c_i \in \cor[[u]]$ such that 
\begin{equation} \label{eq: c-i} 
	c_i(u) = c_{\theta(i)}(-u), \qquad i = \theta(i) \ \Rightarrow \ c_i(u) = 1. 
\end{equation} 
Set
\begin{equation*}
	\widetilde{P}_{ij}(u,v) = P_{ij}(u,v)c_{ij}(u,v), \quad \widetilde{P}_{i}(u) = P_i(u)c_i(u). 
\end{equation*} 

\begin{corollary}
	There is an injective $\tbP_{ \beta }$-algebra homomorphism 
	\eq \label{eq: oklr loc}
	\oaklr \hookrightarrow \skewrrr
	\eneq
	given by 
	\begin{align*} 
		\tau_0 e(\nu) =& \ 
		\left\lbrace
		\begin{array}{ll}
			x_1^{-1}(s_0-1	) e(\nu) \ \qquad \qquad & \textrm{ if } \nu_1 = \theta(\nu_1), \\ 
			\widetilde{P}_{\nu_1}(x_1)s_0e(\nu) \ \qquad \qquad & \textrm{ otherwise, } 
		\end{array}
		\right. \\
		\tau_k e(\nu) =& \ 
		\left\lbrace
		\begin{array}{ll}
			(x_k-x_{k+1})^{-1}(s_k-1)e(\nu) & \textrm{ if } \nu_k=\nu_{k+1}, \\ 
			\widetilde{P}_{\nu_k,\nu_{k+1}}(x_k,x_{k+1})s_k e(\nu) & \textrm{ otherwise, } \\ 
		\end{array} 
		\right. 
	\end{align*} 
	for $1 \leq k \leq n-1$. 
\end{corollary}

\begin{proof}
	This follows directly from Proposition \ref{pro: polrep oklr}. 
\end{proof}

\subsection{PBW theorem} 

In this subsection assume that $Q$ is a coefficient matrix satisfying (M1)--(M3) and $Q'$ a coefficient vector satisfying (V1)--(V2). The algebra $\oaklr$ is filtered with $\deg x_l$, $\deg e(\nu) = 0$ and $\deg \tau_k = 1$. 
We say that $\oaklr$ satisfies the \emph{PBW property} if 
$\gr \oaklr \cong {}^0\mathcal{H}_n^f \ltimes \tbP_{ \beta }$, where ${}^0\mathcal{H}_n^f$ is the (non-affine) nil-Hecke algebra of type $B_n$ (see, e.g., \cite{Kostant-Kumar-nil}). 

For any $w \in \weyl_n$, choose a reduced expression $w = s_{k_1} \hdots s_{k_l}$ and set $\tau_w = \tau_{s_{k_1}} \hdots \tau_{s_{k_l}}$. The definition of $\tau_w$ depends on the choice of reduced expression. 

\begin{prop} \label{thm: oklr PBW}
	Let $n \geq 1$. 
	The following are equivalent: 
	\begin{enumerate}
		\item $\oaklr$ satisfies the PBW property, 
		\item $\oaklr$ is a free $\cor$-module with basis 
		\[ \{ \tau_w x_1^{a_1}\hdots x_n^{a_n} e(\nu) \mid w \in \weyl_n,  (a_1,\hdots,a_n) \in \Z_{\geq0}^n, \nu \in \comp \},\] 
		\item $Q$ and $Q'$ are perfect. 
	\end{enumerate}
\end{prop} 

\begin{proof}
	The proof is a straightforward generalization of the proof of \cite[Theorem 3.7]{Rouquier-2KM}. 
	Let us briefly comment on the new features. 
	Suppose that (2) holds and let $\nu_1 \neq \theta(\nu_1)$. The quadratic relation 
	then implies that 
	\[
	Q_{\theta(\nu_1)}(-x_1)\tau_0e(\nu) = \tau_0^{3}e(\nu) = \tau_0 Q_{\nu_1}(-x_1)e(\nu) = Q_{\nu_1}(x_1)\tau_0e(\nu). 
	\] 
	It follows that 
	\[ (Q_{\theta(\nu_1)}(-x_1) - Q_{\nu_1}(x_1))\tau_0e(\nu) = 0.\] 
	Now (2) implies that $Q_{\theta(\nu_1)}(-x_1) - Q_{\nu_1}(x_1) = 0$, i.e., $Q'$ is self-conjugate. Conversely, if both $Q$ and $Q'$ are perfect, then we can use Proposition \ref{pro: polrep oklr}, 
	with $P_{ij} = Q_{ij}$, $P_{ji} = 1$ $(i < j)$, $P_i = Q_i$ and $P_{\theta(i)} = 1$ ($i < \theta(i)$) for some ordering of $J$, to deduce (2). 
\end{proof} 

\subsection{Orientifold KLR algebras associated to quivers}
\label{ss: quiver klr}

Let $\Gamma = (J,\Omega)$ be a quiver with vertices $J$ and arrows $\Omega$. We assume that $\Gamma$ does not have loops. 
Given an arrow $a \in \Omega$, let $s(a)$ be its source, and $t(a)$ its target. 
If $i, j \in J$, let $\Omega_{ij} \subset \Omega$ be the subset of arrows $a$ such that $s(a) = i$ and $t(a) = j$. Let $a_{ij} = |\Omega_{ij}|$ and abbreviate $\dvec a_{ij} = a_{ij}+a_{ji}$. We assume that $a_{ij} < \infty$ for all $i,j \in J$. 

\begin{definition}  \label{def: contr-inv}
	A (contravariant) \emph{involution} of the quiver $\Gamma$ is a pair of involutions $\theta \colon J \to J$ and $\theta \colon \Omega \to \Omega$ such that: 
	\begin{enumerate}
		\item $s(\theta(a)) = \theta(t(a))$ and $t(\theta(a)) = \theta(s(a))$ for all $a \in \Omega$, 
		\item if $t(a) = \theta(s(a))$ then $a = \theta(a)$. 
	\end{enumerate} 
\end{definition} 

Fix a quiver $\Gamma$ with an involution $\theta$, and two dimension vectors $\beta \in \N[J]^\theta$, $\lambda \in \N[J]$ such that $\Norm{\beta}_{\theta} = n$ and $\lambda(i) = 0$ if $i \in J^\theta$. We call $\lambda$ the \emph{framing dimension vector}. Note that $\lambda$ need not be self-dual. 

Set 
\[ P_{ij}(u,v) = \delta_{i \neq j} (v-u)^{a_{ij}}, \quad P_{i}(u) = \delta_{i \neq \theta(i)} (-u)^{\lambda(i)}\] 
for $i, j \in J$, and define $(Q,Q')$ as in \eqref{eq: P->Q}. Since, by Definition \ref{def: contr-inv}, $a_{ij} = a_{\theta(j)\theta(i)}$, the coefficient matrix $P$ is $\theta$-symmetric and, therefore, $(Q,Q')$ are perfect. 

\begin{definition}
	The KLR algebra associated to $(\Gamma, \beta)$ and the orientifold KLR algebra associated to $(\Gamma, \theta, \beta, \lambda)$ are, respectively, 
	\[ \mathcal{R}^\Gamma(\beta) = \mathcal{R}(J, Q, \beta), \quad {}^\theta\mathcal{R}^\Gamma(\beta;\lambda) = {}^\theta\mathcal{R}(J, Q, Q', \beta).\] 
	We endow these algebras with the following grading: 
	\begin{align*}
		\deg e(\nu) =& \ 0, \\
		\deg x_k =& \ 2, \\ 
		\deg \tau_k e(\nu) =& \ 
		\left\lbrace
		\begin{array}{ll}
			-2 & \textrm{ if } \nu_k=\nu_{k+1}, \\ 
			\dvec a_{\nu_k,\nu_{k+1}} \quad  & \textrm{ otherwise, } \\  
		\end{array} 
		\right. \\
		\deg \tau_0 e(\nu) =& \ 
		\left\lbrace
		\begin{array}{ll}
			-2 & \textrm{ if } \theta(\nu_1) = \nu_{1}, \\ 
			{}^\theta\lambda(\nu_1) \quad  \ & \textrm{ otherwise. } \\ 
		\end{array} 
		\right. 
	\end{align*} 
\end{definition}

Most of the time we will omit $\Gamma$ from the notation, as the choice of quiver is clear from the context. Also note that, by Proposition \ref{pro: polrep oklr}, the algebra $\oklr$ has a faithful polynomial representation on $\tbP_{\beta}^{P,P'}$.

\section{Enomoto-Kashiwara algebra, quantum shuffle modules and Lyndon words} \label{sec: EK}

\subsection{Notation} \label{subsec: notation}

Let $J = \{ \alpha_k \mid k \in \zodd \}$ and equip $\mathsf{Q} = \Z[J]$ with the symmetric bilinear form 
\begin{equation} \label{bilform}
\alpha_k \cdot \alpha_l = \left\{
\begin{array}{l l}
2 & \mbox{if } k=l \\ 
- 1 & \mbox{if } k = l \pm 2 \\
0 & \mbox{otherwise.}  
\end{array} \right.
\end{equation} 
Then $(J,\cdot)$ is the Cartan datum associated to $\g = \mathfrak{sl}_\infty$. 
We identify $J$ with the set of simple roots of the root system $\Phi$ of type $\mathtt{A}_\infty$. 
Then $\proot = \{ \beta_{k,l} \mid k \leq l \in \zodd \}$, where $\beta_{k,l} = \alpha_{k} + \alpha_{k+2} + \cdots + \alpha_{l}$, is a set of positive roots. 
Let $\mathsf{P} = \{ \lambda \in \Q \otimes_{\Z} \mathsf{Q} \mid \lambda \cdot i \in \Z \mbox{ for all } i \in J \}$ be the weight lattice,  $\mathsf{P}_+ = \{ \lambda \in \mathsf{P} \mid \lambda \cdot i \in \Z_{\geq 0} \mbox{ for all } i \in J \}$ the set of dominant integral weights, and $\mathsf{Q}_+ = \N[J]$. 
Given $\beta = \sum_{i \in J} c_i i \in \mathsf{Q}_+$, let $N(\beta) = \frac{1}{2} (\beta \cdot \beta - \sum_{i \in J} c_i i \cdot i)$. 

Let $\theta \colon \mathsf{Q} \to \mathsf{Q}$ be the involution sending $\alpha_k \mapsto \alpha_{-k}$. The bilinear form \eqref{bilform} restricts to $\mathsf{Q}^\theta$. The image of $\Phi$ under the symmetrization map 
\[
\mathsf{Q} \to \mathsf{Q}^\theta, \quad \alpha_k \mapsto \alpha_k + \alpha_{-k} 
\] 
is isomorphic to the unreduced root system $\ttt\Phi$ of type $\mathtt{BC}_\infty$, and the image $\ttt\Phi^+$ of $\Phi^+$ is a set of positive roots for $\ttt\Phi$.

Let $q$ be an indeterminate and set $\Qq = \Q(q)$, $\Ql = \Z[q^{\pm 1}]$. Let \ $\bar{} \ \colon \Qq \to \Qq$ be the \emph{bar involution}, i.e., the $\Q$-algebra involution with $\bar{q} = q^{-1}$. Set 
\[
[n] = (q^n - q^{-n})/(q - q^{-1}), \quad [n]! = [n][n-1]\cdots[1], \quad [2n]!! = [2n][2n-2]\cdots[2]. 
\]
If $A$ is a $\Qq$-algebra, $a \in A$ and $n \in \N$, then $a^{(n)} = a^n/[n]!$. 
For $\nu = \nu_1^{a_1} \hdots \nu_k^{a_k} \in \wor$ with $\nu_j \neq \nu_{j+1}$, set $[\nu]! = [a_1]! \hdots [a_k]!$.

\subsection{The algebras $\ff$ and $\ffd$} 

Let $\ff$ be the $\Qq$-algebra generated by the elements $f_i$ $(i \in J)$ subject to the $q$-Serre relations: 
\[
\sum_{k+l = 1-i\cdot j} (-1)^k f_i^{(k)}f_jf_i^{(l)} = 0 \quad (i \neq j). 
\] 
The algebra $\ff$ is $\mathsf{Q}$-graded with $f_i$ in degree $-i$. We denote by $-\norm{u}$ the $\mathsf{Q}$-degree of a homogeneous element $u \in \ff$. 
Given $\nu = \nu_1 \cdots \nu_n \in \wor$, let $f_\nu = f_{\nu_1} \cdots f_{\nu_n}$. We will use notation of this form  more generally, i.e., given any collection of elements $y_i$ labelled by $i \in J$, we write $y_\nu = y_{\nu_1} \cdots y_{\nu_n}$. 

Kashiwara \cite{Kashiwara-91} introduced $q$-derivations $e'_i, e^*_i \in \End_{\Qq}(\ff)$ characterized by 
\begin{alignat*}{3}
e'_i(f_j) =& \ \delta_{ij}, \quad& e'_i(uv) =& \ e'_i(u)v + q^{-i \cdot \norm{u}} u e'_i(v), \\ 
e^*_i(f_j) =& \ \delta_{ij}, \quad& e^*_i(uv) =& \ q^{-i \cdot \norm{v}}e^*_i(u)v +  u e^*_i(v). 
\end{alignat*}
for all homogeneous elements $u,v \in \ff$.
Both $\{ e'_i \mid i \in J\}$ and $\{ e^*_i \mid i \in J\}$ satisfy the $q$-Serre relations. 

There is a unique non-degenerate symmetric bilinear form $(\cdot, \cdot)$ on $\ff$ such that $(1,1) = 1$ and
$(e'_i(u),v) = (u,f_iv)$ for $u,v \in \ff$ and $i \in J$. This form differs slightly from the form $(\cdot, \cdot)_L$ introduced by Lusztig \cite[Proposition 1.2.3]{lusztig-94} - see \cite[\S 2.2]{Leclerc-04} for the precise relationship. 
Let $\fint$ be the integral form of $\ff$, i.e., the $\Ql$-subalgebra generated by the $f_i^{(k)}$ $(i \in J, k \in \N)$, and let $\fintd = \{ u \in \ff \mid (u,v) \in \Ql \mbox{ for all } v \in \fint \}$ be its dual.

\subsection{Enomoto-Kashiwara algebra}

The subalgebra of $\End_{\Qq}(\ff)$ generated by the $e'_i$ and left multiplication by $f_i$ is called the \emph{reduced $q$-analogue} of $U(\g)$. The generators satisfy the relation 
\[
e'_i f_j = q^{-\alpha_i \cdot \alpha_j} f_j e'_i + \delta_{ij}. 
\] 
Enomoto and Kashiwara \cite{Enomoto-Kashiwara-06} defined a related algebra, which also depends on the involution $\theta$. As it appears this algebra does not have a distinctive name in the literature, we call it the Enomoto--Kashiwara algebra. 

\begin{definition}
The \emph{Enomoto--Kashiwara} algebra $\EK$ is the $\Qq$-algebra generated by $E_i$, $F_i$ and the invertible elements $T_i$ $(i \in J)$ subject to the following relations: 
\begin{itemize}
\item the $T_i$'s commute, 
\item $T_{\theta(i)} = T_i$ for any $i$, 
\item $T_iE_jT_i^{-1} = q^{(i+\theta(i))\cdot j}E_j$ and $T_iF_jT_i^{-1} = q^{-(i+\theta(i))\cdot j}F_j$ for $i,j\in J$, 
\item $E_iF_j = q^{-i\cdot j}F_jE_i + \delta_{ij} +\delta_{\theta(i)j}T_i$ for all $i,j \in J$, 
\item the $E_i$'s and the $F_i$'s satisfy the $q$-Serre relations. 
\end{itemize}
\end{definition}

\begin{prop} \label{pro:EKm}
Let $\lambda \in \mathsf{P}_+$. 
\begin{enumerate} 
\item There exists a $\EK$-module $\EKm$ generated by a non-zero vector $v_\lambda$ such that: 
\begin{enumerate}
\item $E_i v_\lambda = 0$ for any $i \in J$, 
\item $T_i v_\lambda = q^{\ttt\lambda \cdot i} v_\lambda$ for any $i \in J$, 
\item $\{ u \in \EKm \mid E_i u = 0 \mbox{ for any } i \in J \} = \Qq v_\lambda$. 
\end{enumerate} 
\item $\EKm$ is irreducible and unique up to isomorphism. 
\item There exists a unique symmetric bilinear form $( \cdot, \cdot )$ on $\EKm$ such that $(v_\lambda,v_\lambda) = 1$ and $(E_i u, v) = (u, F_i v)$ for any $i \in J$ and $u,v \in \EKm$. It is non-degenerate. 
\item There is a unique endomorphism $\bar{\cdot}$ of $\EKm$, called the bar involution, such that $\overline{v_\lambda} = v_\lambda$ and $\overline{a v} = \bar{a} \bar{v}$, $\overline{F_i v} = F_i \overline{v}$ for $a \in \Qq$ and $v \in \EKm$. 
\item Let $\ttt\widetilde{V}(\lambda)$ be the free $\ff$-module with generator $\tilde{v}_\lambda$ and a $\EK$-module structure given by 
\begin{alignat}{3} 
T_i(u\tilde{v}_\lambda) &= q^{\ttt\lambda \cdot i - (i+\theta(i))\cdot \norm{u}}{u}\tilde{v}_\lambda, \\ 
E_i(u\tilde{v}_\lambda) &= e'_i(u)\tilde{v}_\lambda,\\ \label{eq:Faction}
F_i(u\tilde{v}_\lambda) &= (f_iu + q^{\ttt\lambda \cdot i - i\cdot\norm{u}}{u}f_{\theta(i)})\tilde{v}_\lambda, 
\end{alignat}
for any $i \in J$ and $u \in \ff$. Then the subspace of $\ttt\widetilde{V}(\lambda)$ spanned by the vectors $F_\nu \cdot \tilde{v}_\lambda$ is a $\EK$-submodule isomorphic to $\EKm$. 
\end{enumerate}
\end{prop}

\begin{proof}
See \cite[Proposition 2.11, Lemma 2.15]{Enomoto-Kashiwara-08}. 
\end{proof}

From now on, let us identify $\ff$ with the subalgebra of $\EK$ generated by the $F_i$'s. 
Note that it follows from Proposition \ref{pro:EKm} that $\EKm = \ff \cdot v_\lambda$. 
The module $\EKm$ has a $\mathsf{P}^\theta$-grading: 
\[
\EKm = \bigoplus_{\mu \in P^\theta} \EKm_\mu, 
\]
where $\EKm_\mu = \{ v \in \EKm \mid T_i \cdot v = q^{\mu \cdot i} u\}$. If $v \in \EKm_\mu$, write $\mu_v := \mu$ and $\ttt\norm{v} = \mu_v$. 
The intergral and dual integral forms are defined as $\EKmil = \fint v_\lambda$ and $\EKmiu = \{v \in \EKm \mid (\EKmil, v ) \in \Ql \}$, respectively.

The operators $E_i$ satisfy a kind of ``twisted derivation" property. 

\begin{lemma} \label{lem:qder1}
We have 
\[
E_iy \cdot v = q^{-i\cdot\norm{y}} yE_i \cdot v + (e'_i(y) + q^{-i \cdot \norm{e_{\theta(i)}^*(y)}} e_{\theta(i)}^*(y)T_i) \cdot v
\]
for any $y \in \ff$ and $v \in \EKm$. 
\end{lemma}

\begin{proof}
This is \cite[Lemma 2.9]{Enomoto-Kashiwara-08}. 
\end{proof}

\subsection{Quantum shuffle algebra}

The \emph{quantum shuffle algebra} $\FF$ is the $\mathsf{Q}$-graded $\Qq$-algebra with basis $\wor$, where $\deg_{\mathsf{Q}} \nu = -\norm{\nu}$, and multiplication given by 
\eq \label{shufflealgform}
\nu \circ \nu' = \sum_{w \in \coset{\Norm{\beta}}{\Norm{\beta'}}} q^{-d(\nu,\nu',w)} w \cdot \nu \nu'
\eneq
for $\nu \in J^{\beta}$, $\nu' \in J^{\beta'}$, where $\nu\nu' = i_1 \cdots i_{\Norm{\beta+\beta'}}$ and 
\eq \label{shufflealgformgr}
d(\nu,\nu',w) = \sum_{\substack{k \leq \Norm{\beta} < l,\\ w(k) > w(l)}} i_{w^{-1}(k)} \cdot i_{w^{-1}(l)}. 
\eneq

To $\nu = i_1 \hdots i_k \in \wor$ one associates the $q$-derivation $\partial_\nu = e_{i_1}^* \hdots e_{i_k}^* \in \End_{\Qq}( \ff )$. There is a $\Qq$-linear map
\eq \label{Psi map}
\Psi \colon \ff \longrightarrow \FF, \quad \Psi(u) = \sum_{\nu \in \wor, \norm{\nu} = \norm{u}} \partial_{\nu}(u) \cdot \nu
\eneq
for a homogeneous element $u \in \ff$. 

Let $\ei, \eis \in \End_{\Qq}(\FF)$ be the left and right deletion operators:  
\[
\ei(i_1 \cdots i_k) = \delta_{i,i_1} i_2 \cdots i_k, \quad \eis (i_1 \cdots i_k) = \delta_{i,i_k} i_1 \cdots i_{k-1}, \quad \ei(\varnothing) = \eis(\varnothing) = 0, 
\]
respectively.

\begin{prop}
The map \eqref{Psi map} is an injective $\mathsf{Q}$-graded algebra homomorphism satisfying 
\[
\ei \circ \Psi = \Psi \circ e'_i, \quad \eis \circ \Psi = \Psi \circ e^*_i. 
\]
\end{prop} 

\begin{proof} 
This follows directly from \cite[Lemma 3, Theorem 4]{Leclerc-04}. 
The proof for left deletions is analogous. 
\end{proof} 

We will now consider some anti-automorphisms of $\ff$ and $\FF$. Set 
\begin{equation} \label{eq:twinv}
\sigma \colon \wor \to \wor, \quad \nu \mapsto w_0(\nu), \quad \quad \ttt\sigma \colon \wor \to \wor, \quad \nu \mapsto \ttt w(\nu). 
\end{equation} 
We extend these maps to $\Qq$-linear maps $\sigma \colon \FF \to \FF$ and $\ttt\sigma \colon \FF \to \FF$. We use the same symbols to denote the $\Qq$-linear maps 
\[ \sigma \colon \ff \to \ff, \quad f_\nu \mapsto f_{\sigma(\nu)}, \quad \quad \ttt\sigma \colon \ff \to \ff, \quad f_\nu \mapsto f_{\ttt\sigma(\nu)}, \] 
respectively.  
\begin{lemma}
The maps $\sigma$ and $\ttt\sigma$ are algebra anti-automorphisms satisfying $\sigma \circ \Psi = \Psi \circ \sigma$   and $\ttt\sigma \circ \Psi = \Psi \circ \ttt\sigma$, respectively. 
\end{lemma} 

\begin{proof}
The case of $\sigma$ is \cite[Proposition 6]{Leclerc-04}. The case of $\ttt\sigma$ follows easily from \eqref{shufflealgform} and \eqref{shufflealgformgr}. 
\end{proof}

\subsection{Quantum shuffle module} 

We will now realize the modules $\EKm$ in terms of modules over the shuffle algebra. 

\begin{definition} \label{def: q shuffle module}
We define the \emph{quantum shuffle module} $\tFF$ to be the $\mathsf{P}^\theta$-graded $\Qq$-vector space with basis $\twor$, where $\deg_{\mathsf{P}^\theta} \nu = \ttt\lambda - \ttt\norm{\nu}$, and a right $\FF$-action given by 
\eq \label{sh action} \nu \acts \nu' = \sum_{w \in \tcoset{\Norm{\beta}_\theta}{\Norm{\beta'}}} q^{-d(\nu,\nu',w)} w \cdot \nu\nu' \eneq
for $\nu \in \comp$, $\nu' \in J^{\beta'}$, where 
\[
d(\nu,\nu',w) = \sum_{\substack{1\leq k < l \leq N ,\\ w(k) > w(l)}} i_{w^{-1}(k)} \cdot i_{w^{-1}(l)} + 
\sum_{\substack{1\leq k < l \leq N ,\\ w(-k) > w(l)}} i_{w^{-1}(-  k)} \cdot i_{w^{-1}(l)}
- \sum_{\substack{\Norm{\beta}_\theta < l,\\ w(l) < w(-l)}} \ttt\lambda \cdot i_l, 
\]
with $N = \Norm{\beta}_\theta + \Norm{\beta'}$. 
\end{definition}

\begin{remark}
We have chosen to define $\EKm$ as a left $\EK$-module but $\tFF$ as a right $\FF$-module. This choice is a compromise. On the one hand, we wanted to be consistent with the conventions of \cite{Enomoto-Kashiwara-06, Enomoto-Kashiwara-08}. On the other hand, as shown in \cite{Appel-Przezdziecki-21}, $\EKm$ can be categorified via quantum symmetric pairs, which are, by convention (see, e.g., \cite{kolb-14}), right coideal subalgebras. 
\end{remark}

Let $\mathbf{E}_i \in \End_{\Qq}(\tFF)$ be the right deletion operator: 
\[
\mathbf{E}_i(i_1\cdots i_k) = \delta_{i,i_k} i_1\cdots i_{k-1}, \quad \mathbf{E}_i(\varnothing) = 0. 
\]

\begin{lemma} 
\label{lem:qder2}
Formula \eqref{sh action} defines a right $\FF$-action on $\tFF$. 
Moreover, the endomorphisms $\Ei$ satisfy 
\[
\Ei( v \acts z) = q^{-i\cdot\norm{z}} \Ei(v) \acts z + v \acts \eis(z) + q^{-i \cdot \norm{\mathbf{e}_{\theta(i)}'(z)}+\mu_v\cdot i} v \acts \mathbf{e}'_{\theta(i)}(z). 
\]
\end{lemma}

\begin{proof}
The first statement follows easily from the definitions, so we omit a proof. Let us prove the second statement. It suffices to consider $v$ and $z$ of the form $v = \nu j$ and $z = k\mu l$, for $\nu \in \twor$, $\mu \in \wor$ and $j,k,l \in J$.  Then \eqref{sh action} implies 
\[ v \acts z = \nu j \acts k\mu l = (v \acts k\mu)l + q^{-d(v,z,w)}(\nu \acts z)j + q^{-d(v,z,w')}(v \acts \mu l)\theta(k),\] 
where $w$ transposes $j$ and $z$ while $w'$ sends $k$ to $\theta(k)$ and transposes it with $\mu l$. 
One easily sees that $d(v,z,w) = j\cdot \norm{z}$ and $d(v,z,w') = \theta(k) \cdot \norm{\mathbf{e}_{k}'(z)} - \mu_v(\theta(k))$. Hence 
\[
\Ei( v \acts z) = \delta_{i,l} (v \acts k\mu) + \delta_{i,j}q^{-i\cdot\norm{z}}(\nu \acts z) + \delta_{i,\theta(k)}q^{-i \cdot \norm{\mathbf{e}_{\theta(i)}'(z)}+\mu_v\cdot i}(v \acts \mu l). 
\] 
The statement follows. 
\end{proof}

To $\nu = \nu_1 \cdots \nu_k \in \wor$ one associates the operator ${}^\theta\partial_\nu = E_{\nu_1} \cdots E_{\nu_k} \in \End( \EKm )$. There is a $\Qq$-linear map 
\eq \label{tPsi map}
{}^\theta\Psi \colon \EKm \longrightarrow \tFF, \quad \ttt\Psi(u) = \sum_{\nu \in \twor, \ttt\norm{\nu} = \ttt\norm{u}} \ttt\partial_{\nu}(u) \cdot \sigma(\nu)
\eneq
for a homogeneous element $u \in \EKm$. 
Let us abbreviate 
\[ \mathbf{U} = \Psi(\ff), \quad \EKM(\lambda) = \ttt\Psi(\EKm).\]

\begin{prop} \label{pro:shufflemodulereal}
The map \eqref{tPsi map} is injective, $\mathbf{E}_i \circ \ttt\Psi = \ttt\Psi \circ E_i$, and the diagram
\[ \begin{tikzcd}[ row sep = 0.2cm]
\ff \arrow[r, "\Psi"]  & \FF  \\
\curvearrowright & \curvearrowleft \\
\EKm \arrow[r,"\ttt\Psi"] & \tFF
\end{tikzcd} \]
commutes.   
\end{prop}

\begin{proof} 
The injectivity of $\ttt\Psi$ follows directly from part 1.c of Proposition \ref{pro:EKm}. 
Let $\ttt\Psi' \colon \EKm \to \tFF$ be the map sending $y\cdot v_\lambda \mapsto \varnothing \acts \Phi(\sigma(y))$ for $y \in \ff$. Note that $\ttt\Psi'$ is defined on all of $\EKm$ since $\EKm = \ff \cdot v_\lambda$. We claim that $\ttt\Psi'$ intertwines the actions of $\ff$ and $\FF$, and that 
$\ttt\Psi = \ttt\Psi'$. 
For the first claim, note that \eqref{sh action} implies that 
$\nu \acts i = \nu \circ i + q^{\ttt\lambda(i) - i\cdot \norm{v}} \theta(i) \circ \nu$, for $i \in J$ and $\nu \in \wor$. Hence, by part 5 of Proposition \ref{pro:EKm} and \eqref{eq:Faction}, the first claim follows. 
Lemma \ref{lem:qder1} and Lemma \ref{lem:qder2} imply that $\mathbf{E}_i \circ \ttt\Psi' = \ttt\Psi' \circ E_i$. 
Let $v \in \EKm$ be homogeneous and let $\nu \in \twor$ with $\ttt\norm{v} = \ttt\norm{\nu}$. Let $\gamma_\nu(v)$ be the coefficient of $\sigma(\nu)$ in $\ttt\Psi'(v)$. Then $\gamma_\nu(v) = \mathbf{E}_{\sigma(\nu)} \circ \ttt\Psi'(v) =  \ttt\partial_{\nu}(v)$. Hence $\ttt\Psi = \ttt\Psi'$, which completes the proof. 
\end{proof}

\subsection{$\theta$-good words}

We fix a total order on the set $J$ 
and equip $\wor$ with the corresponding anti-lexicographic order. Both are denoted by $\leq$. 
Given a linear combination $u$ of words, let $\max(u)$ be the largest word appearing in $u$.

\begin{lemma} \label{lem:maxofshuffle}
If $\mu' \leq \mu$, $\nu' \leq \nu$ and $\ttt w(\nu') \leq \ttt w(\nu)$, for $\mu, \mu' \in \twor$ and $\nu, \nu' \in \wor$ (with $\Norm{\mu} = \Norm{\mu'}$ and $\Norm{\nu} = \Norm{\nu'}$), then $\max(\mu' \acts\nu') \leq \max(\mu \acts\nu)$. If any of the former three inequalities is strict, then the last fourth inequality is strict, too. 
\end{lemma} 

\begin{proof} 
If $w \in \tcoset{\Norm{\mu}_\theta}{\Norm{\nu}}$, then the condition in the hypothesis forces $w\cdot\mu'\nu'$ to be smaller than or equal to $w \cdot\mu\nu$. 
\end{proof}

A word $\nu \in \wor$ is called \emph{good} if $\nu = \max(\Psi(x))$ for some homogeneous $x \in \ff$.  
Let $\gooda$ denote the set of good words and $J_+^\beta = \gooda \cap J^\beta$. 
We now define the analogue of good words for quantum shuffle modules. 
\begin{definition}
A word $\nu \in \twor$ is called $\theta$\emph{-good} if $\nu = \max(\ttt\Psi(u))$ for some homogeneous $u \in \EKm$. 
Let $\tgooda$ denote the set of all $\theta$-good words and ${}^\theta J^\beta_+ = \tgooda \cap {}^\theta J^\beta$. 
\end{definition} 

In \cite{Leclerc-04}, a ``monomial" basis  $\{ \mathbf{m}_\nu = \Psi(f_{\sigma(\nu)}) \mid \nu \in \gooda \}$ of $\mathbf{U}$ was constructed. 
An analogous basis exists for $\EKmb$. 

\begin{lemma} \label{lem:monbas}
There is a unique basis of homogeneous vectors $\{ \ttt\mathbf{m}_\nu ^*\mid \nu \in \tgooda \}$ of $\EKmb$ 
such that $\mathbf{E}_{\mu}(\ttt\mathbf{m}_\nu) = \delta_{\mu,\nu}$ for any $\mu$ with $\tnorm{\mu} = \tnorm{\nu}$. 
The adjoint basis is $\{ \ttt\mathbf{m}_\nu = \ttt\Psi(F_{\sigma(\nu)} \cdot v_\lambda) \}$.  
\end{lemma} 
\begin{proof}
The proof is completely analogous to the proof of \cite[Proposition 12]{Leclerc-04}. 
\end{proof}

Let $\mathcal{F}^{\mathrm{fr}}$ be the free associative $\Qq$-algebra generated by $J$ (with multiplication given by concatenation of letters) and let $V^{\mathrm{fr}}$ be its right regular representation. There is an algebra homomorphism 
\[ \Xi \colon \mathcal{F}^{\mathrm{fr}} \to \mathcal{F}, \quad \nu=\nu_1\cdots\nu_k \mapsto \nu_1 \circ \cdots \circ \nu_k = \Psi(f_{\nu}) \] 
and a linear map 
\[ \ttt\Xi_\lambda \colon V^{\mathrm{fr}} \to \EKmb, \quad \nu \mapsto \varnothing \acts \Xi(\nu) = \ttt\mathbf{m}_\nu.\] 
intertwining the actions of $\mathcal{F}^{\mathrm{fr}}$ and $\mathcal{F}$. 
We have the following characterization of $\theta$-good words.

\begin{lemma} \label{lem:tgood-equiv-char} 
The following are equivalent: 
\be
\item $\nu \in \twor$ is $\theta$-good, 
\item $\nu$ cannot be expressed modulo $\ker \ttt\Xi_\lambda$ as a linear combination of words $\mu > \nu$. 
\ee
\end{lemma}   

\begin{proof} 
Let $u \in \EKmb$ and $\nu \in \twor$ satisfy $\tnorm{u} = \tnorm{\nu}$ and $\mathbf{E}_\nu(u) \neq 0$. Part (3) of Proposition \ref{pro:EKm} implies that 
$0 \neq (\mathbf{E}_\nu(u), \varnothing) = (u, \ttt\mathbf{m}_\nu)$. 
If $\nu$ could be expressed modulo $\ker \ttt\Xi_\lambda$ as a linear combination of words $\mu > \nu$, then there would exist a relation of the form 
\eq \label{eq:kerexpr}
 \ttt\mathbf{m}_\nu = \sum_{\mu > \nu} c_\mu \ttt\mathbf{m}_\mu 
\eneq
for some $c_\nu \in \Qq$. 
Hence 
\[
0 \neq \mathbf{E}_\nu(u) =  \sum_{\mu > \nu} c_\mu \mathbf{E}_\mu(u). 
\] 
Therefore, $\mathbf{E}_\mu(u) \neq 0$ for some $\mu > \nu$, which implies that $\mu$ is not $\theta$-good. This proves the implication $(1) \Rightarrow (2)$. 

Conversely, let $\tilde{\tgooda}$ be the set of words in $\twor$ satisfying (2). We have shown that $\tgooda \subseteq \tilde{\tgooda}$. Lemma \ref{lem:monbas} implies that the set $\{ \ttt\mathbf{m}_\nu \mid \nu \in \tilde{\tgooda} \}$ contains a basis of $\EKmb$. Moreover, it is linearly independent. Indeed, if there was a linear relation between words of $\tilde{\tgooda}$, one could express the smallest one in terms of the others and it would not belong to $\tilde{\tgooda}$. 
\end{proof} 

\begin{lemma} \label{lem:tgood-good}
The $\theta$-good words have the following properties. 
\be 
\item If $\nu$ is $\theta$-good and $\nu = \mu_1 \mu_2$, then $\mu_1$ is $\theta$-good. 
\item If $\nu$ is $\theta$-good then $\nu$ is good. 
\ee
\end{lemma} 

\begin{proof}
By Proposition \ref{pro:shufflemodulereal}, $\EKmb$ is stable under the operators $\mathbf{E}_i$. Pick $u \in \EKmb$ with $\max(u) = \nu$. Then $\max(\mathbf{E}_{\mu_2}(u)) = \mathbf{E}_{\mu_2}(\max(u)) = \mu_1$. This proves the first part. Next, suppose that $\nu$ is not good. Then, by \cite[Lemma 21]{Leclerc-04}, we have a relation of the form $\mathbf{m}_\nu = \sum_{\mu > \nu} c_\mu \mathbf{m}_\mu$. Applying both sides to $\varnothing$, we get the relation \eqref{eq:kerexpr}. Hence, by Lemma \ref{lem:tgood-equiv-char}, $\nu$ is not $\theta$-good. This proves the second part. 
\end{proof} 

\subsection{Lyndon words}

A nontrivial word $\nu \in \wor$ is called \emph{Lyndon} if it is smaller than all its proper left factors. Note that our definition uses the opposite of the convention of \cite{Leclerc-04, Kleshchev-Ram-11}, where right factors are used instead. 
Let $\mathcal{L}$ denote the set of Lyndon words and $\mathcal{L}_+ = \mathcal{L} \cap \gooda$ the set of good Lyndon words.

\begin{prop} \label{pro: lyndon}
Lyndon words have the following properties. 
\begin{enumerate}
\item Every word $\nu \in \wor$ has a unique factorization $\nu = \nu^{\lan k \ran}\cdots\nu^{\lan 1\ran}$ into Lyndon words such that $\nu^{\lan1\ran} \geq \cdots \geq \nu^{\lan k \ran}$. 
\item The word $\nu$ is good if and only if each $\nu^{\lan m \ran}$ is good. 
\item The map $\nu \mapsto \norm{\nu}$ yields a bijection $\Lyn \xrightarrow{\sim} \Phi^+$. The induced order on $\Phi^+$ is convex. 
\item Let $\mu \in \mathcal{L} \backslash J$ and write $\mu = \mu_{(1)}\mu_{(2)}$ with $\mu_{(2)}$ a proper Lyndon subword of maximal length. Then $\mu_{(1)} \in \mathcal{L}$. 
\end{enumerate}
\end{prop} 

\begin{proof} 
For part (1), see, e.g., \cite[Theorem 11.5.1]{Lothaire}. 
For parts (2) and (3), see \cite[Propositions 17, 18, 26]{Leclerc-04}. For part (4), see \cite[Lemma 14]{Leclerc-04}. 
\end{proof} 

We call the factorization from part (1) of Proposition \ref{pro: lyndon} the \emph{Lyndon factorization} and the Lyndon words in this factorization \emph{Lyndon factors}. 
We will write it in two ways: $\nu = \plyn{k} \cdots \plyn{1}$ for $\nu^{\lan1\ran} \geq \cdots \geq \nu^{\lan k \ran}$, or $\nu = (\plyn{l})^{n_l} \cdots (\plyn{1})^{n_1}$ for $\nu^{\lan1\ran} > \cdots > \nu^{\lan l \ran}$. 
The factorization from part (4) of Proposition \ref{pro: lyndon} is called the \emph{standard factorization} of a Lyndon word.

Given $x,y \in \FF$, let $[x,y]_q = xy - q^{\norm{x}\cdot\norm{y}}yx$. One defines a map $[ \ ] \colon \mathcal{L} \to \wor$ by induction on the standard factorization: $[i]=i$ for $i \in J$, and $[\nu] = [\nu_{(2)},\nu_{(1)}]_q$ if $\nu = \nu_{(1)}\nu_{(2)}$ is the standard factorization of $\nu$. Next, given $\nu = \nu^{\lan k \ran}  \cdots  \nu^{\lan 1 \ran} \in \wor$, let $[\nu] = [\nu^{\lan k \ran}]  \cdots  [\nu^{\lan 1 \ran}]$. For $\nu \in \gooda$, set 
\[
\mathbf{l}_\nu = \Xi([\nu]) \quad (\nu \in \gooda), \quad \quad \ttt\mathbf{l}_\nu = \ttt\Xi_\lambda([\nu]) \quad (\nu \in \tgooda). 
\]

\begin{prop} \label{pro: lyn basis} 
For any $\nu \in \wor$, we have $\min([\nu]) = \nu$. Moreover, the set $\{ \mathbf{l}_\nu \mid \nu \in \gooda \}$ is a basis of $\mathbf{U}$. 
\end{prop} 

\begin{proof}
See \cite[Propositions 19, 22]{Leclerc-04}.     
\end{proof}

The basis from Proposition \ref{pro: lyn basis} is called the \emph{Lyndon basis}. 

\begin{lemma} \label{lem: trans ml}
The set $\{\ttt\mathbf{l}_\nu \mid \nu \in \tgooda \}$ is a basis of $\EKM(\lambda)$. Moreover, the transition matrix $(c_{\nu\mu})$ from $\{\ttt\mathbf{l}_\nu \mid \nu \in \tgooda \}$ to $\{\ttt\mathbf{m}_\mu \mid \mu \in \tgooda \}$ is triangular with $c_{\nu\nu} = \prod_{i=1}^k (-1)^{\ell(\nu^{\lan k \ran})-1} q^{-N(\norm{\nu^{\lan k \ran}})}$. 
\end{lemma} 

\begin{proof}
By Proposition \ref{pro: lyn basis}, we can write $[\nu] = c_{\nu\nu} \nu + \sum_{\nu < \mu} c_{\nu\mu} \mu$, for some $c_{\nu\mu} \in \Qq$. Applying $\ttt\Xi_\lambda$ to both sides, we get $\ttt\mathbf{l}_\nu = c_{\nu\nu} \ttt\mathbf{m}_\nu + \sum_{\mu > \nu} c_{\nu\mu} \ttt\mathbf{m}_\mu$. By Lemma \ref{lem:tgood-equiv-char}, this can be rewritten as $\ttt\mathbf{l}_\nu = c_{\nu\nu} \ttt\mathbf{m}_\nu + \sum_{\nu < \mu \in \tgooda} c'_{\nu\mu} \ttt\mathbf{m}_\mu$. Hence the transition matrix is triangular. To show the last statement of the lemma,  one uses the same calculation as in \cite[Proposition 30]{Leclerc-04}. 
\end{proof} 

\begin{assum} \label{assumption} 
From now on we assume that we are working with the standard ordering of $J$, i.e., $\alpha_k \leq \alpha_l$ if and only if $k \leq l$. In this case, the map $\ttt \sigma$ in \eqref{eq:twinv} preserves~$\lyn$. 
\end{assum} 

Before stating the next lemma, we need to introduce some notation. 
Given $\mu, \mu' \in \lyn$ with $\norm{\mu} = \beta_{k,l}$, $\norm{\mu'} = \beta_{m,n}$, we write
\[ \mu \subset \mu' \iff m < k \mbox{ and } l < n.\] 

\begin{lemma} \label{lem: delta and comm} 
The following hold: 
\begin{enumerate}
\item If $\nu \in \lyn$ then $\mathbf{l}_\nu$ is a multiple of $\nu$. 
\item If $\nu, \mu \in \lyn$ and $\mu \subset \nu$, then $\nu \circ \mu = \mu \circ \nu$. 
\end{enumerate}
\end{lemma}

\begin{proof}
It suffices to prove the first statement for $\nu = \nu_1 \cdots \nu_l \in \lyn$. 
We proceed by induction on $l$. The base case $l=1$ is clear. Let $\nu = \nu_{(1)}\nu_{(2)}$ be the standard factorization of $\nu$. Since we are working with the standard ordering on $J$, $\nu_{(1)} = i$ for some $i \in J$. By induction, we get that $\mathbf{l}_\nu = \Xi([\nu]) = \Xi([\nu_{(2)}])\circ i - q^{-1} i \circ \Xi([\nu_{(2)}])$ is a multiple of $\nu_{(2)}\circ i - q^{-1} i \circ \nu_{(2)}$. Write $\nu_{(2)} =  j\nu_{(2)}'$ with $j \in J$. Then \eqref{shufflealgform} implies that $\nu_{(2)}\circ i - q^{-1} i \circ \nu_{(2)} = (j(\nu_{(2)}'\circ i) + q i \nu_{(2)}) - q^{-1} (i \nu_{(2)} + q j (i \circ \nu_{(2)}')) = [2] \nu$. This completes the proof of the first statement. 
The second statement now follows directly from \cite[Proposition 30]{Leclerc-04} and \cite[Proposition 3.14(3)]{Enomoto-Kashiwara-08}.
\end{proof}

\begin{definition}
We say that $\nu \in \mathcal{L}$ is \emph{$\theta$-Lyndon} if $\nu \geq \ttt w(\nu)$. Let $\ttt \mathcal{L}$ be the set of $\theta$-Lyndon words, and $\tLyn = \gooda \cap \ttt\mathcal{L}$. Let $\tgoodb$ denote the set of all $\theta$-good words $\mu = \plyn{k} \cdots \plyn{1}$ with $\plyn{k}, \cdots, \plyn{1} \in \tlyn$. Moreover, if $\mu = \plyn{k} \cdots \plyn{1} \in \tgooda$ and $\plyn{k}, \cdots, \plyn{1} \notin \tlyn$, then $\mu$ is called $\theta$-\emph{cuspidal}. Let $\tgoodc$ denote the set of all $\theta$-cuspidal words. 
\end{definition}

\begin{lemma} \label{lem: tgood Lyndon props} 
The $\theta$-good Lyndon words have the following properties. 
\be 
\item If $\nu \in \lyn$ then $\nu \in \mathbf{U}$. 
\item Let $\mu \in \twor$ and $\nu \in \ttt\mathcal{L}$ with $\nu \geq \mu$. Then $\mu\nu = \max(\mu \acts \nu)$. 
\item $\ttt\mathcal{L}_+ \subseteq \mathcal{L}\cap\tgooda$. 
\item Let $\mu \in \tgooda$ and $\nu \in \tLyn$ with $\nu \geq \mu$. Then $\mu \nu \in \tgooda$. 
\item If all of the Lyndon factors of $\nu$ are in $\tLyn$, then $\nu \in \tgooda$. 
\item The map $\nu \mapsto \ttt\norm{\nu}$ yields a bijection $\tlyn \xrightarrow{\sim} \ttt\Phi^+$. 
\ee
\end{lemma} 

\begin{proof} 
Since $\nu$ is good, there exists some homogeneous $x \in \mathbf{U}$ such that $x = \nu + y$ with $\nu$ greater than any word $\mu$ in $y$. 
By Assumption \ref{assumption} and \cite[\S 8.1]{Leclerc-04}, $\nu$ is of the form $\alpha_k\alpha_{k-2}\cdots \alpha_{k-2l}$, which implies that $\nu$ is the smallest word of weight $\norm{\nu}$, so $x = \nu$. 
The proof of (2) is similar to the proof of \cite[Lemma 15]{Leclerc-04}. 
If $\nu \in \ttt \mathcal{L}_+$ then, by definition, $\nu \in \mathcal{L}_+$ and $\nu \geq \ttt w(\nu)$. Hence $\max(\varnothing \acts \nu) = \nu$. 
By part (1), $\nu \in \mathbf{U}$, so $\nu \in \tgooda$. This proves (3). 

Let us prove (4). If $\mu = \varnothing$ then the statement reduces to (3). Otherwise, choose a homogeneous element $\varnothing \neq x \in \EKM(\lambda)$  such that $\mu = \max(x)$. Then, after possible rescaling, $x = \mu + r$, where $r$ is a linear combination of words $< \mu$. We have 
$x \acts \nu = \mu \acts \nu + r \acts \nu$. Part (2) implies that $\max(\mu \acts \nu) = \mu \nu$. It follows from Lemma \ref{lem:maxofshuffle} that $\max(\mu \acts \nu) > \max(r \acts \nu)$. 

Next, we prove (5). Suppose that each factor of $\nu = \plyn{k} \cdots \plyn{1}$ is $\theta$-Lyndon. If $k=1$ then $\nu$ is $\theta$-good by (3). By induction on the number of Lyndon factors, we can assume that $\nu' = \plyn{k} \cdots \plyn{2}$ is $\theta$-good. The statement now follows from (4). Part (6) is clear from the definitions. 
\end{proof}

Given $\nu = \plyn{s} \cdots \plyn{1}, \ \nu' = \plyn{t} \cdots \plyn{s+1} \in \gooda$, let $\sh(\nu,\nu') = \pplyn{\mu}{t}\cdots \pplyn{\mu}{1}$ be the good word obtained by shuffling the Lyndon factors of $\nu$ and $\nu'$ in such a way that $\pplyn{\mu}{t} \leq \cdots \leq \pplyn{\mu}{1}$.

\begin{lemma}
The map
\[
\tgoodc \times \tgoodb \to \tgooda, \quad (\nu, \nu') \mapsto \sh(\nu,\nu'). 
\]
is a well-defined injection. 
\end{lemma} 

\begin{proof}
It is clear the map is injective, so we only have to show that $\sh(\nu,\nu')$ is $\theta$-good. We argue by induction on the number $k$ of Lyndon factors in $\nu' = \plyn{k} \cdots \plyn{1}$. If $k=0$ then $\nu$ is $\theta$-good by assumption. Otherwise, letting $\nu'' = \plyn{k} \cdots \plyn{2}$, we can assume that $\sh(\nu,\nu'')$ is $\theta$-good. If $\plyn{1} \geq \sh(\nu,\nu')$, then $\sh(\nu,\nu') = \sh(\nu,\nu'')\plyn{1}$, and we conclude that $\sh(\nu,\nu') \in \tgooda$ from part (4) of Lemma \ref{lem: tgood Lyndon props}. 

If $\plyn{1} < \sh(\nu,\nu')$ then we require the following generalization of part (4) of Lemma \ref{lem: tgood Lyndon props}: given $a \in \tgooda$ and $b \in \tLyn$ with $b < a$, we have $\sh(a,b) \in \tgooda$. The old proof carries over except that instead of invoking part (2) of Lemma \ref{lem: tgood Lyndon props}, we need to show that $\max(a \acts b) = \sh(a,b)$. Without loss of generality, we may assume $a$ is Lyndon. Since $b \geq \ttt w(b)$, we have $\max(a \acts b) = \max(a \circ b)$. Let us write $a = a_n \cdots a_1$ and $b = b_m \cdots b_1$. Since $a_n \geq \cdots \geq a_1 > b_1$, it follows that $\max(a \circ b) = ba$. 
\end{proof}

Given $\beta \in \mathsf{Q}^\theta_+$, let $\tkpf(\beta)$ denote the number of ways to write $\beta$ as a sum of roots in $\ttt\Phi^+$. 

\begin{prop} \label{prop: dim of V modules tkpf} 
If $\lambda = 0$ then: (i) $\ttt\mathcal{L}_+ = \mathcal{L}\cap\tgooda$, and (ii) $\tgooda = \tgoodb$. Hence $\dim_q \EKM_\beta = \tkpf(\beta)$. 
\end{prop} 

\begin{proof}
Let $S$ be the set of all words $\nu = \plyn{k} \cdots \plyn{1}$ with $\nu^{\lan1\ran} \geq \cdots \geq \nu^{\lan k \ran}$ and each $\plyn{i} \in \tgooda$. Lemma \ref{lem:monbas} and part (5) of Lemma \ref{lem: tgood Lyndon props} imply that $\{ \ttt\mathbf{m}_\nu \mid \nu \in S \}$ is contained in the monomial basis $\{ \ttt\mathbf{m}_\nu \mid \nu \in \tgooda  \}$ of $\EKM$. Let $\EKM' \subseteq \EKM$ be the span of the former. By construction, the generating series of the dimensions of the homogeneous components 
of $\EKM'$ is equal to $\prod_{\beta \in \ttt \Phi^+} \frac{1}{1-\exp \beta}$. On the other hand, it follows from \cite[Theorem 4.15]{Enomoto-Kashiwara-08} that this is also the generating series of the dimensions of the homogeneous components of $\EKM$. Hence $\EKM' = \EKM$. The statement follows.
\end{proof} 

\begin{remark}
Instead of appealing to \cite[Theorem 4.15]{Enomoto-Kashiwara-08} in the proof of Proposition \ref{prop: dim of V modules tkpf}, one could alternatively use the categorification theorem \cite[Theorem 8.31]{VV-HecB} (cited as Theorem \ref{thm: VV EKmod categ} below), together with the geometric realization of orientifold KLR algebras from \cite{VV-HecB} and the classification of isomorphism classes of symplectic/orthogonal representations of symmetric quivers from \cite{Derksen-Weyman}. Indeed, this approach appears promising in generalizing the construction of bases for $\ttt\mathbf{V}(\lambda)$ to the $\lambda \neq 0$ case. 
\end{remark}

\subsection{Symmetric words} 

A word $\nu \in \tlyn$ is called \emph{symmetric} if ${}^\theta w(\nu) = \nu$ and \emph{non-symmetric} otherwise. 
Given $\nu \in \tgooda$, let 
 $\nu_\theta$ be the word obtained from $\nu$ by deleting its symmetric Lyndon factors and $\nu^\theta$ the word obtained by deleting the non-symmetric ones. 
We say that $\nu \in \tgooda$ is \emph{symmetric} if $\nu = \nu^\theta$. 
For each $k\geq1$, let $\xi_k$ be the unique symmetric word in $\tlyn$ with $\norm{\xi_k} = \beta_{-2k+1,2k-1}$. 

\begin{lemma} \label{lem: sym aux} 
Let $\nu \in \tlyn$. If $\nu < \xi_k$ then $\xi_{k+1}$ is a subword of $\nu$. 
Hence $\xi_k > \xi_l$ if and only if $k < l$. 
\end{lemma}

\begin{proof}
The statement follows immediately from part (6) of Lemma \ref{lem: tgood Lyndon props}. 
\end{proof} 

\begin{assum} \label{assumption2} 
From now on until the end of \S \ref{sec: EK} we assume that $\lambda = 0$. We abbreviate $\tFFz = \tFFz(0)$ and $\EKmz = \EKmz(0)$.  
\end{assum} 

\begin{lemma} \label{lem: smallest words} 
Suppose that $\nu \in \tgooda$ is symmetric or $\nu \in \tlyn$. Then $\nu$ is the smallest word in $\ttt J_+^{\ttt|\nu|}$. 
\end{lemma} 

\begin{proof}
Abbreviate $\beta = \ttt|\nu|$. First assume that $\nu \in \tgooda$ is symmetric. 
Let $\nu = \plyn{k}\cdots\plyn{1}$ be its Lyndon factorization. Suppose that there exists a word $\mu = \pplyn{\mu}{l} \cdots \pplyn{\mu}{1} \allowbreak \in \ttt J_+^{\ttt|\nu|}$ with $\mu < \nu$. Then, as explained before Lemma 4.1 in \cite{Melancon}, there is an $a$ such that $\pplyn{\mu}{b} = \plyn{b}$ for $b < a$ and $\pplyn{\mu}{a} < \plyn{a}$. Hence $\plyn{a} > \pplyn{\mu}{a} \geq \hdots \geq \pplyn{\mu}{l}$. Write $\bar{\nu} = \pplyn{\nu}{k} \cdots \pplyn{\nu}{a}$ and $\bar{\mu} = \pplyn{\mu}{l} \cdots \pplyn{\mu}{a}$. 

Since $\plyn{a}$ is symmetric, we have $\plyn{a} = \xi_d$ for some $d \geq 1$. 
By Proposition \ref{prop: dim of V modules tkpf} and Lemma \ref{lem: sym aux}, $\xi_{d+1}$ is a subword of each $\pplyn{\mu}{i}$ $(i \geq a)$. 
In particular, each 
$\pplyn{\mu}{i}$ contains $\alpha_{\pm(2d-1)},\alpha_{\pm(2d+1)}$.   
Hence, if we write $\ttt\norm{\bar{\nu}} = \ttt\norm{\bar{\mu}} = \sum_{i \in \nodd} c_i (\alpha_i  + \alpha_{-i})$, then $c_{2d+1} = c_{2d-1}$. On the other hand, since each $\plyn{i}$ $(i \geq a)$ is a symmetric good Lyndon word smaller than $\plyn{a}$, Lemma \ref{lem: sym aux}  implies that each $\plyn{i}$ contains $\plyn{a}$ as a subword. 
Hence $c_{2d+1} < c_{2d-1}$, which is a contradiction. 

Secondly, assume that $\nu \in \tlyn$. We may assume $\nu$ is not symmetric. In that case, observe that if $\ttt\norm{\mu} = \ttt\norm{\nu}$ for some $\mu \in \tgooda$, then $\norm{\mu} = \norm{\nu}$. The result now follows from \cite[Lemma 5.9]{Kleshchev-Ram-11}. 
\end{proof} 

\subsection{PBW and canonical bases} 

Let us first recall some basic facts about PBW bases. For the moment let us restrict $(J,\cdot)$ to a finite Cartan subdatum of type $A_m$. 
By \cite[Proposition 26]{Leclerc-04}, the anti-lexicographic order $\plyn{1} > \hdots >  \plyn{N}$ on the set of good Lyndon words induces, via the bijection from part (3) of Proposition \ref{pro: lyndon}, a convex order $\beta_1 > \hdots > \beta_N$ on the set of positive roots.  
This convex order arises from a unique reduced decomposition $w_0 = s_{i_N}\cdots s_{i_1}$ in the usual way:
$\beta_N = \alpha_{i_N}$, $\beta_{N-1} = s_{i_N}(\alpha_{i_{N-1}})$, $\cdots$, $\beta_1 = s_{i_N} \cdots s_{i_2}(\alpha_1)$. 
Let $P_{\plyn{k}} = T''_{i_N,1}\cdots T''_{i_{k+1},1}(f_{i_k})$, where $T''_{i,1}$ is the braid group operation from \cite[\S 37.1]{lusztig-94} with $e=-1$ and $\upsilon_i = q$. Set $P_{\plyn{k}}^{(l)} = \frac{1}{[l]!} P_\nu^l$ and, 
given $\nu = (\plyn{N})^{l_N} \cdots (\plyn{1})^{l_1} \in \gooda$, let $P_\nu = P_{\plyn{N}}^{(l_N)} \cdots P_{\plyn{1}}^{(l_1)}$ and $\mathbf{P}_\nu = \Psi(P_\nu)$. Taking an appropriate limit $m \to \infty$, \cite[Proposition 41.1.4]{lusztig-94} implies that $\{ P_\nu \mid \nu \in \gooda \}$ is an $\Ql$-basis of $\fint$.

Next, given $\nu \in \tlyn$, let 
\[
P_{\nu}^{[n]} = \left\{ 
\begin{array}{l l}
P_{\nu}^{(n)} & \mbox{ if } \nu \mbox{ is not symmetric,} \\ \\
\frac{1}{[2n]!!} P_{\nu}^n & \mbox{ if } \nu \mbox{ is symmetric. } 
\end{array} \right. 
\] 
Given $\nu =  (\plyn{l})^{n_l} \cdots (\plyn{1})^{n_1} \in \tgooda$, define 
\[
\ttt P_{\nu} = \sigma \left(\prod_{1 \leq i \leq l} P_{\plyn{i}}^{[n_i]} \right) \cdot v_0, \quad \ttt\mathbf{P}_\nu = \ttt\Psi(P_\nu).
\] 

\begin{prop} 
The set $\{ \ttt P_\nu \mid \nu \in \tgooda \}$ is an $\Ql$-basis of $\EKmilz$. 
\end{prop} 

\begin{proof}
See \cite[Lemma 5.1]{Enomoto-Kashiwara-08}. Note that the weaker statement that $\{ \ttt P_\mu \}$ is a $\Qq$-basis of $\EKmilz$ follows from Lemma \ref{lem: trans ml} and part (1) of Lemma \ref{lem: pbw conjugate} below. 
\end{proof}

We call $\{ \ttt \mathbf{P}_\nu \mid \nu \in \tgooda \}$ the PBW basis of $\EKlow$. 
By \cite[Proposition 30]{Leclerc-04}, for any $\nu \in \gooda$, there exists $\kappa_\nu = \overline{\kappa_\nu} \in \Ql$ with $\mathbf{l}_\nu = \kappa_\nu \mathbf{P}_\nu$. Since we are working with the standard ordering of $J$,  \cite[Proposition 56]{Leclerc-04} implies that $\kappa_\nu = 1$ for any $\nu \in \lyn$. 
If $\nu = (\plyn{l})^{n_l} \cdots (\plyn{1})^{n_1} \in \tgooda$ then $\kappa_\nu = \prod_{i=1}^l [n_i]!$. Set 
\[
\tkap = \kappa_\nu \cdot \prod_{\substack{i=1\\ \plyn{i} \text{symm.}}}^l \prod_{j=1}^{n_i} (q^j + q^{-j}) = \prod_{\substack{i=1\\ \plyn{i} \text{symm.}}}^l [n_i]!! \cdot \prod_{\substack{i=1\\ \plyn{i} \text{non-symm.}}}^l [n_i]!. 
\]

\begin{lemma} \label{lem: pbw conjugate}
Let $\nu \in \tgooda$. 
\begin{enumerate} 
\item $\ttt{\mathbf{l}}_\nu = \tkap \tpbw{\nu}$ and $\tkap = \overline{\tkap} \in \Ql$.  
\item We have
\[
\overline{\tpbw{\nu}} = \tpbw{\nu} + \sum_{\mu > \nu} d_{\nu\mu} \tpbw{\mu}
\]
for some $d_{\nu\mu} \in \Ql$. 
\end{enumerate} 
\end{lemma} 

\begin{proof} 
The first part follows directly from the definitions. 
Let $A_{\mathbf{P}}$, $A_{\mathbf{m}}$ and $A$ be the transition matrices between $\{ \tpbw{\nu} \}$ and $\{ \overline{\tpbw{\nu}} \}$, $\{ \mathbf{m}_\nu \}$ and $\{ \overline{\mathbf{m}_\nu} \}$, as well as $\{ \tpbw{\nu} \}$ and $\{ \mathbf{m}_\nu \}$, respectively. By definition, $A_{\mathbf{m}} = \id$. 
Hence 
$A_{\mathbf{P}} = \overline{A} A^{-1}$. Lemma \ref{lem: trans ml} implies that $\overline{A}$ and $A^{-1}$ are both lower triangular, with eigenvalues $\overline{\tkap}$ and $\tkap^{-1}$. Part (1) now implies that $A_{\mathbf{P}}$ is indeed lower unitriangular. Since $\{ \tpbw{\nu} \}$ form an $\Ql$-basis of $\EKlow$ and $\overline{\EKlow} = \EKlow$, we have $d_{\nu\mu} \in \Ql$. 
\end{proof}

\begin{theorem} \label{thm: can vs pbw} 
There is a unique $\Ql$-basis $\{ \tcan{\nu} \mid \nu \in \tgooda \}$ of $\EKlow$, called the \emph{canonical basis},  such that 
\[
\tcan{\nu} = \tpbw{\nu} + \sum_{\mu > \nu} c_{\nu\mu} \tpbw{\mu}, 
\] 
$c_{\nu\mu} \in q\Z[q]$ and $\overline{\tcan{\nu}} = \tcan{\nu}$. Moreover, 
$(\tcan{\nu},\tcan{\mu})_{q=0} = \delta_{\nu,\mu}$. 
\end{theorem} 

\begin{proof}
The proof is an application of a standard argument, see, e.g., \cite[\S 7.10]{Lusztig-can1}. 
\end{proof} 

\begin{remark} 
Theorem \ref{thm: can vs pbw} also appears in \cite{Enomoto-Kashiwara-08} as Theorem 5.5. The proof in \emph{loc. cit.} is somewhat different from ours, in particular, it does not involve shuffle modules. 
\end{remark}

Let $\{ \ttt \mathbf{P}_\nu^* \mid \nu \in \tgooda \}$ and $\{ \ttt \mathbf{b}_\nu^* \mid \nu \in \tgooda \}$ be the  bases of $\EKup$ dual (with respect to the bilinear form $( \cdot, \cdot )$) to the PBW and the canonical bases of $\EKlow$, respectively. 

\begin{corollary} \label{cor: max of dual can} 
We have 
\eq 
\tdcan{\nu} = \tdpbw{\nu} + \sum_{\mu < \nu} (\tdcan{\nu},\tpbw{\mu}) \tdpbw{\mu}. 
\eneq 
Hence $\max(\tdcan{\nu}) = \nu$ and the coefficient of $\nu$ in $\tdcan{\nu}$ is $\ttt\kappa_\nu$. In particular, if $\nu \in \tlyn$ or $\nu$ is symmetric, then $\tdcan{\nu} = \tdpbw{\nu}$. 
\end{corollary} 

\begin{proof}
The proof is analogous to \cite[Proposition 40]{Leclerc-04}. The last statement follows from Lemma \ref{lem: smallest words}. 
\end{proof}

\subsection{Standard and costandard basis} 

Given $\nu = (\plyn{l})^{n_l} \cdots (\plyn{1})^{n_1} \in \tgooda$, let 
\[
\Delta_\nu = q^{-s(\nu)}(\plyn{l})^{\circ n_l} \circ \cdots \circ (\plyn{1})^{\circ n_1}, \quad 
{}^\theta \Delta_\nu = q^{-\ttt s(\nu)} \varnothing \acts \Delta_\nu, 
\] 
where 
\eq \label{eq: s ts}
s(\nu) = \sum_{i=1}^l n_i(n_i - 1)/2, \quad \ttt s(\nu) = \sum_{\substack{i=1\\ \plyn{i} \text{symm.}}}^l n_i. 
\eneq

\begin{lemma} \label{lem: standard char}
If $\nu \in \tgooda$ then: $\Delta_\nu = \Delta_{\nu^\theta} \circ \Delta_{\nu_\theta}$,  $\max({}^\theta \Delta_\nu) = \nu$ and the coefficient of the word $\nu$ in ${}^\theta \Delta_\nu$ equals $\tkap$. 
\end{lemma} 

\begin{proof} 
We prove the first statement by induction one the number $k$ of Lyndon factors in the Lyndon factorization of $\nu^\theta$. If $k=0$ the claim is obvious. Next, suppose that there are $k+1$ Lyndon factors in $\nu^\theta$ and let $\xi_m$ be the smallest. If $\xi_m$ is also the smallest word in the standard factorization of $\nu$ then, by induction, we are done. Otherwise, let $\mu$ be a Lyndon factor of $\nu$ with $\mu < \xi_m$. Since $\mu \in \tlyn$, Lemma \ref{lem: sym aux} implies that $\xi_{m} \subset \mu$. By Lemma \ref{lem: delta and comm}, we conclude that $\mu \circ \xi_m = \xi_m \circ \mu$. It now follows by induction that $\Delta_\nu = \Delta_{\nu^\theta} \circ \Delta_{\nu_\theta}$. 

We now prove the last two statements by induction on the number $k$ of Lyndon factors in $\nu$. The base case $k=0$ is trivial. Let $\nu' = \plyn{k} \cdots \plyn{2}$. Lemma \ref{lem:tgood-good} implies that $\nu' \in \tgooda$. Hence, by induction, $\max(\ttt\Delta_{\nu'}) = \nu'$. Since $\lambda = 0$, we have $\plyn{1} \in \tlyn$, and so $\plyn{1} \geq \ttt w(\plyn{1})$. It follows from Lemma \ref{lem:maxofshuffle} and part (2) of Lemma \ref{lem: tgood Lyndon props} that $\max(\ttt\Delta_{\nu}) = \max(\nu' \acts \plyn{1}) = \nu$. By induction, we may also assume that $\dim_q ({}^\theta \Delta_{\nu'})_{\nu'} = \ttt\kappa_{\nu'}$. Let us call the result of applying $w \in \tcoset{\Norm{\nu'}_\theta}{\Norm{\plyn{1}}}$ to $\nu$ a $\theta$-shuffle. It is easy to see that the $\theta$-shuffles equal to $\nu$ are precisely those arising from one of the $n_1$ (resp.\ $2n_1$) standard insertions of $\plyn{1}$ between words equal to $\plyn{1}$ in $\nu'$ if $\plyn{1}$ is not symmetric (resp.\ is symmetric). We conclude that $\dim_q ({}^\theta \Delta_\nu)_\nu = \tkap$ from the fact that the transposition of two words equal to $\plyn{1}$ appears in the shuffle action with the coefficient $q^{-2}$. 
\end{proof} 

Given $\nu \in \tgooda$ with $\nu = \plyn{k} \cdots \plyn{1}$, let 
\[
\tnabla_\nu = q^{-\ttt s(\nu)-t(\nu)}\varnothing \acts (\ttt w(\plyn{k}) \circ \cdots \circ \ttt w(\plyn{1})), 
\]
where $t(\nu)$ is the degree of an element $\tau_w$, with $w$ the longest minimal length coset representative with respect to the parabolic subgroup of $\weyl_n$ defined by the decomposition of $\nu$ into Lyndon words (cf.\ \cite[\S 2.3]{Lauda-Vazirani}). 

Recall that we have fixed the standard order $\leq$ on $J$ and equipped $\wor$ with the anti-lexicographic order $\leq$. Let $\leq'$ denote both the opposite order on $J$ and the induced lexicographic order on $\wor$. 
Given a linear combination $u$ of words, let $\max'(u)$ be the largest word appearing in $u$ with respect to $\leq'$.  

\begin{lemma} \label{lem: costandard char} 
We have: $\max'(\tnabla_\nu) = \nu$ and  the coefficient of the word $\nu$ in $\tnabla_\nu$ equals~$\tkap$. 
\end{lemma}

\begin{proof}
The proof is an easy modification of the last paragraph of the proof of Lemma \ref{lem: standard char}. 
\end{proof}

\section{Finite-dimensional representation theory of orientifold KLR algebras} \label{sec: rep th}

We again let $\lambda$ be arbitrary until \S \ref{sec:oklr irr gd}, where we make the restriction $\lambda = 0$. 

If $A$ is a graded algebra, let $\gmodv{A}$ be the category of all graded left $A$-modules, with degree-preserving module homomorphisms as morphisms. If $M, N$ are graded $A$-modules, let $\hhom_A(M,N)_n$ denote the space of all homomorphisms homogeneous of degree $n$, and $\Hom_A(M,N) = \bigoplus_{n \in \Z} \hhom_A(M,N)_n$. Let $M\{n\}$ denote the module obtained from $M$ by shifting the grading by $n$. Let $\pmodv{A}$ denote the full subcategory of finitely generated graded projective modules, and $\fmodv{A}$ the full subcategory of graded finite dimensional modules. Given any of these abelian categories $\mathcal{C}$, we denote its Grothendieck group by $[ \mathcal{C} ]$. 

We consider (orientifold) KLR algebras associated to the $\mathtt{A}_\infty$ quiver $\Gamma = (J,\Omega)$, with $J$ as in \S \ref{subsec: notation} and $\Omega$ the standard linear orientation, as well as the involution $\theta$ from \S \ref{subsec: notation}. 
Let $\mathds{1}$ and ${}^\theta\mathds{1}$ denote the regular representations (in degree zero) of the trivial algebras $\klrv{0}$ and $\oklrv{0}$, respectively. 
For a fixed $\lambda \in \N[J]$, set 
\[
\gmodv{\mathcal{R}} = \soplus_{\beta \in \N[J]} \gmodv{\klr}, \quad \gmodv{{}^\theta \mathcal{R}(\lambda)} = \soplus_{\beta \in \N[J]^{\theta}} \gmodv{\oklr}. 
\] 
We use analogous notation for direct sums of categories of finite dimensional and finitely generated projective modules.

\subsection{Reminder on categorification via KLR algebras} 

Basic information about the representation theory of KLR algebras, including the definitions of the Khovanov--Lauda pairing $(\cdot , \cdot) \colon \pmodv{\klr} \times \fmodv{\klr} \to \mathcal{A}$ and the dualities $P \mapsto P^\sharp$ on $\pmodv{\mathcal{R}}$ and $M \mapsto M^\flat$ on $\fmodv{\mathcal{R}}$, can be found in, e.g., \cite{Khovanov-Lauda-1}, \cite[\S 3]{Kleshchev-Ram-11} or \cite[\S 7]{VV-HecB}. Since these definitions and the notations are standard, we will not explicitly recall them. If $M \in \gmodv{\klr}$ and $\nu \in \word$, we call $M_\nu = e(\nu) M$ the $\nu$-weight space of $M$.

Let us recall the definition of the convolution product of modules over KLR algebras. 
Let $\beta, \beta' \in \N[J]$ with $\Norm{\beta} = n$ and $\Norm{\beta'} = n'$. Set 
\[ e_{\beta,\beta'} = \sum_{\stackrel{\nu \in J^{\beta+\beta'}}{ \nu_1 \cdots \nu_n \in J^\beta}} e(\nu) \in \klrv{\beta+\beta'}.\]
There is a non-unital algebra homomorphism 
\eq \label{eq: induction inc klr}
\iota_{\beta,\beta'} \colon \klrvv{\beta}{\beta'} := \klrv{\beta} \otimes \klrv{\beta'} \to \klrv{\beta + \beta'} 
\eneq
given by $e(\nu) \otimes e(\mu) \mapsto e(\nu\mu)$ for $\nu \in J^\beta$, $\mu \in J^{\beta'}$ and 
\begin{alignat}{3} \label{eq: induc1}
	x_l \otimes 1 \mapsto& \ x_{l} e_{\beta,\beta'}, \quad& 1 \otimes x_{l'} \mapsto& \ x_{m+l'}e_{\beta,\beta'} &\qquad& (1 \leq l^{(')} \leq n^{(')}), \\ \label{eq: induc2}
	\tau_k \otimes 1 \mapsto& \ \tau_{k}e_{\beta,\beta'}, \quad& 1 \otimes \tau_{k'} \mapsto& \ \tau_{m+k'}e_{\beta,\beta'} &\qquad& (1 \leq k^{(')} < n^{(')}). 
\end{alignat} 
Let $M$ be a graded $\klrv{\beta}$-module and $N$ be a graded $\klrv{\beta'}$-module. Their \emph{convolution product} is defined as 
\[
M \circ N = \klrv{\beta+\beta'}e_{\beta,\beta'}\otimes_{\klrvv{\beta}{\beta'}}(M \otimes N). 
\] 
It descends to a product on $ [\pmodv{\mathcal{R}}]$ and $ [\fmodv{\mathcal{R}}]$. 

The embedding \eqref{eq: induction inc klr} generalizes to an embedding 
\eq \label{eq: long ind inc klr}
\iota_{\underline{\beta}} \colon \klrv{\underline{\beta}} := \klrv{\beta_1} \otimes \cdots \otimes \klrv{\beta_m} \to \klrv{\norm{\underline{\beta}}}
\eneq
for any $\underline{\beta} \in (\N[J])^m$. The embedding \eqref{eq: long ind inc klr} gives rise to a triple of adjoint functors $(\Ind_{\underline{\beta}},\Res_{\underline{\beta}}, \Coind_{\underline{\beta}})$ between categories of graded modules. 

As explained in \cite[\S 2.2]{Khovanov-Lauda-1} and \cite[\S 3.6]{Kleshchev-Ram-11}, convolution with the class of (an appropriate graded shift of) the polynomial representation $P(i^{(n)})$ of the nil-Hecke algebra $\mathcal{R}(ni)$ yields an $\Ql$-module homomorphism 
\[
\theta_i^{(n)} = - \circ [P(i^{(n)})] \colon [\pmodv{\klr}] \to [\pmodv{\klrv{\beta + ni}}]. 
\]

Let us recall the fundamental categorification theorem from \cite[\S 3]{Khovanov-Lauda-1} (see also \cite[Theorem 4.4]{Kleshchev-Ram-11}. 

\begin{theorem}[Khovanov-Lauda]
There exists a unique pair of adjoint (with respect to Lusztig's form on~$\ff$ and the Khovanov--Lauda pairing) $\mathsf{Q}$-graded $\Ql$-linear isomorphisms 
\[ 
\gamma \colon \fint  \xrightarrow{\sim} [\pmodv{\mathcal{R}}], \quad 
\gamma^* \colon [\fmodv{ \mathcal{R}}] \xrightarrow{\sim}  \fintd
\]
such that $\gamma(1) = [\triv]$ and $\gamma(x f_i^{(n)}) = \theta_i^{(n)}(\gamma(x))$ for all $x \in \fint$. These isomorphisms intertwine: (i) multiplication in $\ff$ with the convolution product, (ii) comultiplication in $\ff$ with restriction functors, and (iii) the bar involution on $\ff$ with the involutions $-^\sharp$ and $-^\flat$. 
\end{theorem} 

\subsection{Categorification via orientifold KLR algebras} 

We recall some fundamental definitions and results concerning orientifold KLR algebras from \cite[\S 8]{VV-HecB}. We refer the reader to \emph{loc.\ cit.} for a detailed exposition.

Let $\beta \in \N[J]^\theta$ and $\beta' \in \N[J]$ with $\Norm{\beta}_\theta = n$ and $\Norm{\beta'} = n'$. Set 
\[ {}^\theta e_{\beta,\beta'} = \sum_{\substack{\nu \in {}^\theta J^{\beta+{}^\theta\beta'}, \ \nu_1\hdots\nu_n \in {}^\theta J^\beta \\ \nu_{n+1}\hdots\nu_{n+n'} \in J^{\beta'}}} e(\nu) \in \oklrv{\beta+{}^\theta\beta'}.\] 
There is an injective non-unital algebra homomorphism 
\eq \label{eq: induction inc}
\ttt\iota_{\beta,\beta'} \colon \oklrvv{\beta}{\beta'} := \oklr \otimes \klrv{\beta'} \to \oklrv{\beta+{}^\theta\beta'}
\eneq
given by formulae \eqref{eq: induc1}-\eqref{eq: induc2} (with $\nu \in \comp$ and $e_{\beta,\beta'}$ replaced by ${}^\theta e_{\beta,\beta'}$) and $\tau_0 \otimes 1 \mapsto \tau_0 {}^\theta e_{\beta,\beta'}$. 
The \emph{convolution action} of $N \in \gmodv{\klrv{\beta'}}$ on $M \in \gmodv{\oklr}$ is defined as 
\[
M \acts N =  \oklrv{\beta+{}^\theta\beta'}{}^\theta e(\beta,\beta')\otimes_{\oklrvv{\beta}{\beta'}}(M \otimes N). 
\] 

\begin{prop}
The category $\gmodv{\mathcal{R}}$ is monoidal with product $\circ$ and unit $\mathds{1}$. Moreover, there is a right monoidal action (see, e.g., \cite{Davydov}) of $\gmodv{\mathcal{R}}$ on $\gmodv{{}^\theta \mathcal{R}(\lambda)}$ via $\acts$. 
\end{prop}

\begin{proof}
It is routine to check that the conditions in the definition of a monoidal action are satisfied. 
\end{proof} 

The embedding \eqref{eq: induction inc} generalizes to an embedding 
\eq \label{eq: long ind inc}
\ttt\iota_{\underline{\beta}} \colon \oklrvv{\beta_0}{\underline{\beta}} := \oklrv{\beta_0} \otimes \klrv{\beta_1} \otimes \cdots \otimes \klrv{\beta_m} \to \oklrv{\beta_0+\ttt\norm{\underline{\beta}}}
\eneq
for any $\beta_0 \in \N[J]^\theta$ and $\underline{\beta} \in (\N[J])^m$. The embedding \eqref{eq: long ind inc} gives rise to a triple of adjoint functors $(\ttt\Ind_{\beta_0,\underline{\beta}},\ttt\Res_{\beta_0,\underline{\beta}}, \ttt\Coind_{\beta_0,\underline{\beta}})$ between categories of graded modules. 

\begin{lemma} \label{lem: coind} 
Let $M_0 \in \fmodv{\oklr}$ and $M_i \in \fmodv{\klrv{\beta_i}}$. Then, up to a grading shift, we have 
\[
\textstyle \ttt\Coind_{\beta_0,\underline{\beta}}(M_0 \otimes \bigotimes M_i) 
\cong \ttt\Ind_{\beta_0,\theta(\underline{\beta})}(M_0 \otimes \bigotimes M_i^\dagger) 
\cong \ttt\Coind_{\beta_0,\norm{\underline{\beta}}} (M_0 \otimes (\Coind_{\underline{\beta}}(\bigotimes M_i)), 
\]
where $\theta(\underline{\beta}) = (\theta(\beta_1),\cdots,\theta(\beta_m))$ and $-^\dagger$ is the twist defined below Lemma \ref{lem: klr isos for twist}. 
\end{lemma}

\begin{proof}
The proof is analogous to that of \cite[Theorem 2.2]{Lauda-Vazirani}. 
\end{proof}

Let $\beta_0 \in \N[J]^\theta$ and $\beta_1, \beta_2 \in \N[J]$. Define 
\[
M_1 \lcirc M_2 = \Coind_{\beta_1,\beta_2} (M_1 \otimes M_2), \quad M_0 \lacts M_1 = \ttt\Coind_{\beta_0,\beta_1} (M_0 \otimes M_1),  
\]
for $M_i$ as in Lemma \ref{lem: coind}. 
\begin{corollary}
The category $\gmodv{\mathcal{R}}$ is also monoidal with product $\hat{\circ}$ and unit $\mathds{1}$. Moreover, there is a monoidal action of $\gmodv{\mathcal{R}}$ on $\gmodv{{}^\theta \mathcal{R}(\lambda)}$ via $\lacts$. 
\end{corollary}

The functors $P \mapsto P^\sharp = \Hom_{\oklrn{m}}(P,\oklrn{m})$ and $M \mapsto M^\flat = \Hom_{\cor}(P,\cor)$ on $\pmodv{\oklrn{m}}$ and $\fmodv{\oklrn{m}}$, respectively, descend to $\Ql$-antilinear involutions on the corresponding Grothendieck groups. We also have an analogue of the Khovanov-Lauda pairing
\[
(\cdot , \cdot) \colon [\pmodv{\oklr}] \times [\fmodv{\oklr}] \to \mathcal{A}, \quad ([P],[M]) \mapsto \gdim (P^\omega \otimes_{\oklr} M), 
\] 
where $P^\omega$ is the twist of $P$ by the anti-involution \eqref{anti-invo}. 

Moreover, set $\oklrn{m} = \bigoplus_{\Norm{\beta}_\theta = n} \oklr$ and ${}^\theta e_{m,\beta'} = \oplus_{\Norm{\beta}_\theta = m} {}^\theta e_{\beta,\beta'}$. 
Abbreviate $\ttt\Ind_{m,i}^{m+1} = \oklrn{m+1} \otimes_{\oklrn{m,i}} -$ and $\ttt\Coind_{m,i}^{m+1} = \Hom_{\oklrn{m,i}}(\oklrn{m+1},-)$, with $\oklrn{m,i} = \oklrn{m}\otimes \mathcal{R}(i)$. 
Setting 
\begin{alignat*}{3}
F_i(P) \ &=& \ \ttt\Ind_{m,i}^{m+1} (P \otimes P(i)), \quad& \ E_i(P) \ &=& \  L(i) \otimes_{\mathcal{R}(i)} \ttt e_{m-1,i}P \\ 
F^*_i(M) \ &=& \ \ttt\Coind_{m,i}^{m+1} (M \otimes L(i)), \quad& \ E^*_i(M) \ &=& \  \ttt e_{m-1,i} M, 
\end{alignat*}
defines exact functors  
\[
\begin{tikzcd} 
\pmodv{\oklrn{m}} \arrow[rr, bend left = 15, "F_i"] & &  \pmodv{\oklrn{m+1}} \arrow[ll,  bend left = 15, "E_i"]
\end{tikzcd}
\]
\[
\begin{tikzcd} 
\fmodv{\oklrn{m}} \arrow[rr, bend left = 15, "F_i^*"] & &  \fmodv{\oklrn{m+1}} \arrow[ll,  bend left = 15, "E_i^*"]
\end{tikzcd}
\]
commuting with the dualities $-^\sharp$ and $-^\flat$. 
We will use the same notation for the induced operators on the corresponding Grothendieck groups.

We now recall the main theorem \cite[Theorem 8.31]{VV-HecB} on the categorification of modules over the Enomoto-Kashiwara algebra. 

\begin{theorem}[Varagnolo-Vasserot] \label{thm: VV EKmod categ} 
The operators $F_i, E_i$ (resp.\ $F_i^*, E_i^*$) define a representation of $\EK$ on $\Qq \otimes_{\Ql} [\pmodv{{}^\theta \mathcal{R}(\lambda)}]$ (resp.\ $\Qq \otimes_{\Ql} [\fmodv{{}^\theta \mathcal{R}(\lambda)}]$). Moreover, there exists a unique pair of adjoint $\mathsf{P}^\theta$-graded $\Ql$-linear isomorphisms 
\[
\tcat \colon  \EKmil \xrightarrow{\sim} [\pmodv{{}^\theta \mathcal{R}(\lambda)}], \quad 
\tcat^* \colon [\fmodv{{}^\theta \mathcal{R}(\lambda)}] \xrightarrow{\sim} \EKmiu
\] 
which, upon base change to $\Qq$, become isomorphisms of $\EK$-modules. 
They intertwine the bar involution on $\EKm$ with the involutions $-^\sharp$ and $-^\flat$. 
\end{theorem}

If $M \in \gmodv{\oklr}$ and $\nu \in \comp$, we call $M_\nu = e(\nu) M$ the $\nu$-weight space of $M$. 
The \emph{character} of a $\oklr$-module $M$ is 
$\tchq(M) = \sum_\nu \gdim(e(\nu)M) \cdot \nu \ \in \tFF.$ 
This gives rise to an $\Ql$-linear map $\tchq \colon  [\fmodv{{}^\theta \mathcal{R}(\lambda)}] \to \tFF$. We call $\max(\tchq(M))$ (if it exists) the highest weight of $M$. 

\begin{corollary} \label{cor: VV EKmod categ} 
The following triangle commutes: 
\[ \begin{tikzcd}[column sep = 1.25 em, row sep = 1 em] 
 & \left[\fmodv{{}^\theta \mathcal{R}(\lambda)}\right] \arrow[ld, "\ttt\gamma^*", swap] \arrow[rd, "\tchq"] & \\
\EKmiu \arrow[rr,"\ttt\Psi"] && \tFF
\end{tikzcd} \] 
The map $\tchq$ is injective and $\tchq(M \acts N) = \tchq(M) \acts \chq(N)$. 
\end{corollary}

\begin{proof}
The proof is analogous to that of \cite[Theorem 4.4(3)]{Kleshchev-Ram-11}. 
\end{proof}

\subsection{Reminder on KLR representation theory} 

An irreducible $\klr$-module $L$ is called \emph{cuspidal} if $\max(\chq(L)) \in \lyn$, i.e., its highest weight is a good Lyndon word. By \cite[Proposition 8.4]{Kleshchev-Ram-11}, for each $\nu \in \lyn$, there exists a unique cuspidal irreducible $\klrv{\norm{\nu}}$-module $L(\nu)$. 

Let $\nu = (\plyn{l})^{n_l} \cdots (\plyn{1})^{n_1} \in \good$. The corresponding standard and costandard modules are, respectively,  
\[ \Delta(\nu) = L(\plyn{l})^{\circ n_l} \circ \cdots \circ L(\plyn{1})^{\circ n_1}\{s(\nu)\}, \quad 
\nabla(\nu) = L(\plyn{l})^{\circ n_l} \lcirc \cdots \lcirc L(\plyn{1})^{\circ n_1}\{s(\nu)\}, 
\] 
with $s(\nu)$ as in \eqref{eq: s ts}. 

\begin{theorem}[Kleshchev-Ram, McNamara] \label{thm: klr rep theory} 
Let $\nu \in \good$. Then: 
\be
\item The standard $\klr$-module $\Delta(\nu)$ has an irreducible head $L(\nu)$, and the costandard module $\nabla(\nu)$ has $L(\nu)$ as its socle. 
\item The highest weight of $L(\nu)$ is $\nu$, and $\dim_q L(\nu)_\nu = \kappa_\nu$. 
\item $L(\nu) = L(\nu)^\flat$. 
\item $\{L(\nu) \mid \nu \in \good \}$ is a complete and irredundant set of irreducible graded $\klr$-modules up to isomorphism and degree shift. 
\item If $L(\mu)$ is a composition factor of $\Delta(\nu)$ (resp.\ $\nabla(\nu)$), then $\mu \leq \nu$ (resp.\ $\mu \leq' \nu$). Moreover, $L(\nu)$ appears in $\Delta(\nu)$ and $\nabla(\nu)$ with multiplicity one. 
\item If $\nu = \mu^n$ for a good Lyndon word $\mu$, then $\Delta(\nu) = L(\nu)$. 
\ee
\end{theorem}

\begin{proof}
See \cite[Theorem 7.2]{Kleshchev-Ram-11}, \cite[Theorem 3.1]{McNamara-klr1}. 
\end{proof}

\subsection{Orientifold KLR: irreducibles and global dimension} 
\label{sec:oklr irr gd}

Now assume $\lambda = 0$. 

\begin{lemma} \label{lem: sym simple irred}
If $\nu \in \tgooda$ is symmetric, then $\tL(\nu) = \ttriv \acts L(\nu)\{ \ttt s(\nu)\}$ is irreducible. The highest weight of $\tL(\nu)$ is $\nu$, $\tchq \tL(\nu) = \tdcan{\nu}$ and $\dim_q \tL(\nu)_\nu = \ttt\kappa_\nu$. 
\end{lemma} 

\begin{proof} 
It follows from Lemma \ref{lem:maxofshuffle}, Lemma \ref{lem: smallest words} and Corollary \ref{cor: VV EKmod categ} that all composition factors of $\tL(\nu)$ have highest weight $\nu$. We know from part (2) of Theorem \ref{thm: klr rep theory} that $\max(\chq(L(\nu))) = \nu$ and $\dim_q L(\nu)_\nu = \kappa_\nu$. The last part of Corollary \ref{cor: VV EKmod categ}, together with an argument analogous to that in the last paragraph of the proof of Lemma \ref{lem: standard char}, then shows that the highest weight of $\tL(\nu)$ is $\nu$ and $\dim_q \tL(\nu)_\nu = \ttt\kappa_\nu$. 

Let $\beta = \tnorm{\nu}$. By Theorem \ref{thm: VV EKmod categ}, $\tchq \tL(\nu) \in \ttt\mathbf{V}_{\Ql,\beta}^{\mathrm{up}}$. Since $\{ \ttt \mathbf{b}_\mu^* \mid \mu \in {}^\theta J^\beta_+ \}$ is an $\Ql$-basis of $\ttt\mathbf{V}_{\Ql,\beta}^{\mathrm{up}}$, we have $\tchq \tL(\nu) = \sum_{\mu \in {}^\theta J^\beta_+} c_\mu \ttt \mathbf{b}_\mu^*$ for some $c_\mu \in \Ql$. By Corollary \ref{cor: max of dual can}, $\max(\ttt \mathbf{b}_\mu^*) = \mu$, and, by Lemma \ref{lem: smallest words}, $\nu$ is the smallest word in ${}^\theta J^\beta_+$. Hence $c_\mu = 0$ unless $\mu = \nu$. Comparing the coefficients of $\nu$ in $\chq L(\nu)$ and $\ttt \mathbf{b}_\nu^*$, we conclude that $c_\nu = 1$. The irreducibility of $\tL(\nu)$ follows directly from the equality $\tchq \tL(\nu) = \tdcan{\nu}$. 
\end{proof} 

For $\nu \in \tgood$, let 
\[
\tdelta(\nu) = \ttriv \acts \Delta(\nu), \quad \tnabla(\nu) = \ttriv \lacts  \nabla(\nu). 
\]

\begin{lemma} \label{lem: qdim tdelta nu}
Let $\nu \in \tgooda$. Then: $\Delta(\nu) = \Delta(\nu^\theta) \circ \Delta(\nu_\theta)$, $\max(\tchq {}^\theta \Delta(\nu)) = \nu$ and $\dim_q ({}^\theta \Delta(\nu))_\nu = \tkap$.  
\end{lemma} 

\begin{proof} 
The proof of the first statement is analogous to the proof of the first statement of Lemma \ref{lem: standard char}. Using the inductive argument and the notation from that proof, one observes that $\mu \xi_m$ is the lowest good word of weight $|\mu \xi_m|$. Part (5) of Theorem \ref{thm: klr rep theory} then implies that $L(\mu) \circ L(\xi_m) = \Delta(\mu \xi_m) = L(\mu \xi_m) = \nabla(\mu \xi_m) = L(\xi_m) \circ L(\mu)$, allowing the induction to proceed. 

Since $\dim_q L(\mu) = 1$, for all $\mu \in \lyn$ (see \cite[\S 8.4]{Kleshchev-Ram-11}), we have $\tchq(\ttt\Delta(\nu)) = \ttt\Delta_\nu$. The second and third statements now follow from the second and third statements of Lemma \ref{lem: standard char}. 
\end{proof} 

\begin{theorem} \label{thm: main result}
Let $\nu \in \tgood$. Then: 
\be 
\item The standard $\oaklr$-module $\tdelta(\nu)$ has an irreducible head $\tL(\nu)$ and the costandard $\oaklr$-module $\tnabla(\nu)$ has $\tL(\nu)$ as its socle.  
\item The highest weight of $\tL(\nu)$ is $\nu$, and $\dim_q \tL(\nu) = \ttt\kappa_\nu$ . 
\item $\ttt L(\nu) = \ttt L(\nu)^\flat$. 
\item $\{\tL(\nu) \mid \nu \in \tgood \}$ is a complete and irredundant set of irreducible graded $\oaklr$-modules up to isomorphism and degree shift. 
\item If $\tL(\mu)$ is a composition factor of $\tdelta(\nu)$ (resp.\ $\tnabla(\nu)$), then $\mu \leq \nu$ (resp.\ $\mu \leq' \nu$). Moreover, $\tL(\nu)$ appears in $\tdelta(\nu)$ and $\tnabla(\nu)$ with multiplicity one. 
\item If $\nu$ is a Lyndon word or $\nu = \nu^\theta$, then $\tdelta(\nu)=\tL(\nu)$ is irreducible. 
\ee
\end{theorem} 

\begin{proof}
The structure of the proof is similar to \cite[Theorem 7.2]{Kleshchev-Ram-11} (see also \cite[Theorem 3.1]{McNamara-klr1}). Let us explain the main points. If $\nu_\theta = (\plyn{l})^{n_l} \cdots (\plyn{1})^{n_1}$, let $\beta_0 = \tnorm{\nu^\theta}$, $\underline{\beta} = (n_l\norm{\plyn{l}}, \cdots, n_1\norm{\plyn{1}})$ and abbreviate $\ttt\Res_{\nu} = \ttt\Res_{\beta_0,\underline{\beta}}$, $\ttt\mathcal{R}_\nu = \ttt\mathcal{R}(\beta_0,\underline{\beta})$. 
Also abbreviate $\tL(\vec{\nu}) = \tL(\nu^\theta) \otimes L(\plyn{l})^{\circ n_l} \otimes \cdots \otimes L(\plyn{1})^{\circ n_1}\{s(\nu_\theta)\}$. Let $L$ be an irreducible $\oaklr$-module in the head of $\tdelta(\nu)$. 
By adjunction and the first part of Lemma \ref{lem: qdim tdelta nu}, $\Hom_{\oaklr}(\tdelta(\nu),\tdelta(\nu)) = \Hom_{\ttt\mathcal{R}_\nu}(\tL(\vec{\nu}), \ttt\Res_{\nu} \tdelta(\nu))$ and $0 \neq \Hom_{\oaklr}(\tdelta(\nu), L) = \Hom_{\ttt\mathcal{R}_\nu}(\tL(\vec{\nu}), \ttt\Res_{\nu} L)$. 
Hence we get the commutative diagram 
\[ \begin{tikzcd}[column sep = 1.25 em, row sep = 1 em] 
\tL(\vec{\nu}) \arrow[r, hookrightarrow] \arrow[d, equal] & \ttt\Res_{\nu} \tdelta(\nu) \arrow[r, hookrightarrow] \arrow[d, twoheadrightarrow] & \tdelta(\nu) \arrow[d, twoheadrightarrow] \\
\tL(\vec{\nu}) \arrow[r, hookrightarrow] & \ttt\Res_{\nu} L \arrow[r, hookrightarrow] & L. 
\end{tikzcd} \] 
The injectivity of the two arrows on the left follows from the fact that the $\ttt\mathcal{R}_\nu$-module $\tL(\vec{\nu})$ is irreducible, which is implied by part (5) of Theorem \ref{thm: klr rep theory} and Lemma \ref{lem: sym simple irred}. Part (2) of Theorem \ref{thm: klr rep theory}, Lemma \ref{lem: sym simple irred} and Lemma \ref{lem: qdim tdelta nu} also imply that $\dim_q \tL(\vec{\nu})_\nu = \tkap = \dim_q \tdelta(\nu)_\nu$. Hence $\dim_q L_\nu = \tkap$ as well, implying that the head of $\tdelta(\nu)$ is irreducible. This proves (1) in the case of  standard modules as well as (2). Note that the modules $\tL(\nu)$ we have thus constructed are pairwise non-isomorphic since they have different highest weights. 

Next, (3) follows from \cite[Proposition 2]{VV-HecB} and the fact that $\tkap$ is bar-invariant (Lemma \ref{lem: pbw conjugate}). Part (4) follows from Proposition \ref{prop: dim of V modules tkpf}, Theorem \ref{thm: VV EKmod categ} and the fact that we have constructed $\tkpf(\beta)$ non-isomorphic irreducible graded $\oaklr$-modules $\{ \tL(\nu) \mid \nu \in \tgood \}$. Next, we return to (1) in the case of costandard modules. 
An analogous argument to that in the case of standard modules, using Lemma \ref{lem: costandard char} and the adjunction between restriction and coinduction now shows that $\tnabla(\nu)$ has an irreducible socle with highest weight $\nu$, which, by (4), must be isomorphic to $\tL(\nu)$. 
Part (5) follows immediately from the facts that $\nu = \max(\tchq(\tdelta(\nu))) = \max'(\tchq(\tnabla(\nu)))$ and $\dim_q \tdelta(\nu)_\nu = \dim_q \tnabla(\nu)_\nu = \dim_q \tL(\nu)_\nu$. 
Next, part (6) follows from Lemma \ref{lem: smallest words}  and (5). 
\end{proof}

\begin{corollary} \label{cor: fin gl dim}
As a graded algebra, $\oaklr$ has global dimension $\Norm{\beta}_\theta$. 
\end{corollary} 

\begin{proof} 
The proof is analogous to \cite[Theorem 4.7]{McNamara-klr1}. For the sake of simplicity, we ignore the grading shifts. Since $\lambda = 0$, $\tgooda$ contains no $\theta$-cuspidal words. Let $\nu, \mu \in \tgood$. If $\nu_\theta = (\plyn{l})^{n_l} \cdots (\plyn{1})^{n_1}$, let $L(\vec{\nu}) = L(\nu^\theta) \otimes L(\plyn{l})^{\circ n_l} \otimes \cdots \otimes L(\plyn{1})^{\circ n_1}$. Also write $\underline{\beta} = (\norm{\nu^\theta},n_l\norm{\plyn{l}}, \cdots, n_1\norm{\plyn{1}})$. Then, Lemma \ref{lem: coind} and adjunction between induction and restriction imply that 
\[
\Ext^i_{\oaklr}(\tnabla(\nu),\tdelta(\mu)) = \Ext^i_{\mathcal{R}(\underline{\beta})}(L(\vec{\nu}), \Res_{\underline{\beta}} \tdelta(\mu)), 
\]
which, by \cite[Theorem 4.7]{McNamara-klr1} is zero for $i > \Norm{\beta}_\theta$. The rest of the proof is exactly the same as in \cite{McNamara-klr1}. 
\end{proof}


\bibliographystyle{myamsalpha}

\providecommand{\bysame}{\leavevmode\hbox to3em{\hrulefill}\thinspace}
\providecommand{\MR}{\relax\ifhmode\unskip\space\fi MR }
\providecommand{\MRhref}[2]{%
  \href{http://www.ams.org/mathscinet-getitem?mr=#1}{#2}
}
\providecommand{\href}[2]{#2}

\end{document}